\documentclass[letterpaper,11pt,twoside,reqno]{amsart}
\usepackage{times}
\usepackage{amsmath}
\usepackage{amssymb}
\usepackage{amsfonts}
\usepackage{amscd}
\usepackage{amsthm}
\usepackage{amsxtra}
\usepackage{stmaryrd}

\usepackage[all]{xy}  \CompileMatrices
\usepackage{mathrsfs}
\usepackage{nicefrac}
\usepackage{graphicx}
\usepackage{color}
\usepackage[raggedright,IT,hang]{subfigure}
\usepackage{wrapfig}
\usepackage{enumerate}

\usepackage{tikz}
\usetikzlibrary{matrix,arrows}

\renewcommand{\paragraph}[1]{\par\vspace{1ex}\noindent #1}

\newcommand{\volker}[1]{{#1}}

\newcommand{\old}[1]{}
\newcommand{\vold}[1]{}
\newcommand{\new}[1]{{#1}}

\theoremstyle{plain}
\newtheorem{theorem}{Theorem}
\newtheorem{proposition}[theorem]{Proposition}
\newtheorem*{proposition*}{Proposition}
\newtheorem{corollary}[theorem]{Corollary}
\newtheorem*{corollary*}{Corollary}
\newtheorem{lemma}[theorem]{Lemma}

\newtheorem*{theorem*}{Theorem}
\newtheorem*{lemma*}{Lemma}
\newtheorem*{conjecture*}{Conjecture}

\newtheorem*{question*}{Question}

\theoremstyle{definition}
\newtheorem{definition}[theorem]{Definition}

\newtheorem*{exercise*}{Exercise}

\theoremstyle{remark}

\newtheorem*{remark*}{Remark}
\newtheorem{remsTh}[theorem]{Remarks}

\newcommand{\subclass}[1]{}

\newcommand{\enumTi}[1]{\renewcommand{\theenumi}{#1}}

\newcommand{\alphenumi}{\enumTi{\alph{enumi}}}

\newcommand{\romenumi}{\enumTi{\roman{enumi}}}

\renewcommand{\em}{\sl}

\DeclareMathOperator{\aff}{aff}
\DeclareMathOperator{\conv}{conv}

\newcommand{\Order}{O}

\newcommand{\sabs}[1]{{\lvert{#1}\rvert}}

\newcommand{\NN}{\mathbb{N}}

\newcommand{\RR}{\mathbb{R}}

\newcommand{\ZZ}{\mathbb{Z}}

\newlength{\algotabbingwidth}
\newcommand{\setDef}[2]{\{{#1}\,:\,{#2}\}}
\newcommand{\ints}[1]{[{#1}]}
\newcommand{\symGr}[1]{\mathfrak{S}({#1})}
\newcommand{\altGr}[1]{\mathfrak{A}({#1})}
\newcommand{\charFct}[1]{\mathbf{1}_{{#1}}}
\newcommand{\charVec}[1]{\chi({#1})}
\newcommand{\zeroVec}[1]{\mathbf{0}_{{#1}}}

\newcommand{\ident}[1]{\mathbf{I}_{{#1}}}
\DeclareMathOperator{\isoOp}{iso}
\newcommand{\isoGr}[2]{\isoOp_{{#1}}({#2})}
\DeclareMathOperator{\POp}{P}
\DeclareMathOperator{\matchOp}{match}
\newcommand{\PMatch}[2]{\POp_{\matchOp}^{{#1}}({#2})}
\DeclareMathOperator{\sptOp}{spt}
\newcommand{\PSpt}[1]{\POp_{\sptOp}({#1})}

\DeclareMathOperator{\cyclOp}{cycl}
\newcommand{\PCycl}[2]{\POp_{\cyclOp}^{{#1}}({#2})}
\DeclareMathOperator{\oddOp}{odd}

\newcommand{\match}[2]{\mathcal{M}^{{#1}}({#2})}
\newcommand{\oddints}[1]{[{#1}]_{\oddOp}}
\newcommand{\cycl}[2]{\mathcal{C}^{{#1}}({#2})}
\newcommand{\scalProd}[2]{\langle{#1},{#2}\rangle}
\DeclareMathOperator{\trOp}{t}
\newcommand{\transpose}[1]{{#1}^{\trOp}}
\DeclareMathOperator{\im}{im}
\newcommand{\isoms}[1]{\mathcal{O}({#1})}
\newcommand{\row}[2]{{#1}_{{#2},\star}}
\newcommand{\slackMap}[1]{\Delta\!^{#1}}
\newcommand{\slackMapRestr}[1]{\Delta\!^{#1}_{\perp}}
\newcommand{\slackRep}[1]{\Delta(#1)}
\DeclareMathOperator{\linealOp}{lineal}
\newcommand{\lineal}[1]{\linealOp({#1})}

\newcommand{\upop}{\varkappa}

\begin{document}

\title{Symmetry Matters for Sizes of Extended Formulations}%
\author{Volker Kaibel}%
\address{Volker Kaibel}
\email{kaibel@ovgu.de}%
\author{Kanstantsin Pashkovich}%
\address{Kanstantsin Pashkovich}
\email{pashkovi@imo.math.uni-magdeburg.de}%
\author{Dirk Oliver Theis}%
\address{Dirk Oliver Theis}%
\email{dirk.theis@ovgu.de}%
\thanks{Research by K. Pashkovich has been supported by the \emph{International Max Planck Research School for Analysis, Design and Optimization in Chemical and Biochemical Process Engineering Magdeburg}\\
	An extended abstract of this work has been published as~\cite{KPT10}: 
	V.~Kaibel, K.~Pashkovich, and D.\,O.~Theis, \emph{Symmetry matters for the sizes
	of extended formulations}, in: F.~Eisenbrand and B.~Shepherd (eds.), Integer Programming and Combinatorial Optimization (Proc. IPCO XIV), volume 6080 of LNCS, pages 135–-148. Springer, 2010.
}
% \subjclass[2000]{Primary 52B12 Secondary 52B11}
%\subjclass[2000]{\dots, \dots}
%\keywords{Keywords}

\date{\today}%

%\dedicatory{}

\begin{abstract}

In 1991, Yannakakis~\cite{Yan91} proved that  no  symmetric extended formulation for the matching  polytope of the complete graph~$K_n$ with~$n$ nodes  has a number of variables and constraints that is bounded subexponentially in~$n$. Here, symmetric means  that the formulation remains invariant under all permutations of the nodes of~$K_n$. 
It was also conjectured in~\cite{Yan91}  that ``asymmetry does not help much,'' but no corresponding result for general extended formulations has been found so far. In this paper we show that for the polytopes associated with the matchings in~$K_n$ with $\lfloor\log n\rfloor$ edges there are non-symmetric extended formulations of polynomial size, while nevertheless no  symmetric extended formulations of polynomial size exist. We furthermore prove similar statements for the polytopes associated with cycles of length~$\lfloor\log n\rfloor$. Thus,  with respect to the question for  smallest possible extended formulations, in general symmetry requirements  may matter a lot. Compared to the extended abstract~\cite{KPT10}, this paper does not only contain proofs that had been ommitted there, but it also presents slightly generalized and sharpened lower bounds. 
\end{abstract}
\maketitle

%%%%%%%%%%%%%%%%%%%%%%%%%%%%%%%%%%%%%%%%%%%%%%%%%%%%%%%%%%%%%%%%%%%%%%%%%%%%%%%%%%%%%%%%%%%%%%%%%%%%%%%%%%%%%%%%%%%%%%%%%%%%%%%%%%%%%%%%%%%%%%%%%%%%%%
%%%%%%%%%%%%%%%%%%%%%%%%%%%%%%%%%%%%%%%%%%%%%%%%%%%%%%%%%%%%%%%%%%%%%%%%%%%%%%%%%%%%%%%%%%%%%%%%%%%%%%%%%%%%%%%%%%%%%%%%%%%%%%%%%%%%%%%%%%%%%%%%%%%%%%
%%%%%%%%%%%%%%%%%%%%%%%%%%%%%%%%%%%%%%%%%%%%%%%%%%%%%%%%%%%%%%%%%%%%%%%%%%%%%%%%%%%%%%%%%%%%%%%%%%%%%%%%%%%%%%%%%%%%%%%%%%%%%%%%%%%%%%%%%%%%%%%%%%%%%%

% yannakakis/kmatch-intro.tex

\section{Introduction}

Linear Programming techniques have proven to be extremely fruitful for  combinatorial optimization problems
 with respect to both  structural analysis and the design of algorithms. In this context, the paradigm  is to represent the problem by a polytope~$P\subseteq\RR^m$ whose vertices correspond to the feasible solutions of the problem in such a way that the objective function can be expressed by a linear functional $x\mapsto\scalProd{c}{x}$ on~$\RR^m$ (with some $c\in\RR^m$). If one succeeds in finding a description of~$P$ by means of linear constraints, then algorithms as well as structural results from Linear Programming can be exploited. In many cases, however, the polytope~$P$ has exponentially  (in~$m$) many facets, thus~$P$ can only be described by exponentially many inequalities. Also it may be that the inequalities needed to describe~$P$ are too complicated to be identified. 

In some of these cases one may  find an \emph{extended formulation for~$P$}, i.e.,  a  (preferably small and simple) description by linear constraints of another polyhedron $Q\subseteq\RR^d$ in some higher dimensional space  that projects to~$P$ via some (simple) affine
 map $p:\RR^d\rightarrow\RR^m$ with $p(y)=Ty+t$ for all $y\in\RR^d$ (and some $T\in\RR^{m\times d}$, $t\in\RR^m$). As we have  
$\max\setDef{\scalProd{c}{x}}{x\in P}=\max\setDef{\scalProd{\transpose{T}c}{y}}{y\in Q}+\scalProd{c}{t}$ 
for each $c\in\RR^m$, one can solve linear optimization problems over~$P$ by solving linear optimization problems over~$Q$ in this case. 

As for a guiding example, let us consider the spanning tree polytope
\begin{equation*}
	\PSpt{n}=\conv\setDef{\charVec{T}\in\{0,1\}^{E_n}}{T\subseteq E_n\text{ spanning tree of }K_n}\,,
\end{equation*} 
where~$K_n=(\ints{n},E_n)$ denotes the complete graph with node set $\ints{n}=\{1,\dots,n\}$ and edge set $E_n=\setDef{\{v,w\}}{v,w\in \ints{n}, v\ne w}$, and $\charVec{A}\in\{0,1\}^B$ is the \emph{characteristic vector} of the subset $A\subseteq B$ of~$B$, i.e., for all $b\in B$, we have $\charVec{A}_b=1$ if and only if $b\in A$. Thus, $\PSpt{n}$ is the polytope associated with the bases of the graphical matroid of~$K_n$, and we have (see~\cite{Edm71})
\begin{multline}\label{eq:DescrPSpt}
	\PSpt{n}=\{x\in\RR_+^{E_n}\,:\, x(E_n)=n-1,\\
	           x(E(S))\le |S|-1\text{ for all }S\subseteq\ints{n}, 2\le |S|\le n-1\}\,,
\end{multline}
where $\RR_+^{E_n}$ is the nonnegative orthant of~$\RR^E_n$, we denote  by $E(S)$ the subset of all edges with both nodes in~$S$, and $x(F)=\sum_{e\in F}x_e$ for $F\subseteq E_n$.
This linear description of~$\PSpt{n}$ has an exponential (in~$n$) number of constraints, and as all the inequalities define pairwise different facets, none of them is redundant.   

The following  much smaller extended formulation for~$\PSpt{n}$ (with $\Order(n^3)$ variables and constraints)
appears in~\cite{CCZ10} (and a  similar one in~\cite{Yan91}, where it is attributed to~\cite{Mar87}).
Let us introduce additional variables $z_{e,v,u}$ for all  $e\in E_n$, $v\in e$, and $u\in \ints{n}\setminus e$. While each spanning tree $T\subseteq E_n$ is represented by its characteristic vector $x^{(T)}=\charVec{T}$ in~$\PSpt{n}$, in the extended formulation it will be represented by the vector $y^{(T)}=(x^{(T)},z^{(T)})$ with 
$z^{(T)}_{e,v,u}=1$ (for $e\in E_n$, $v\in e$, $u\in \ints{n}\setminus e$) if $e\in T$ and~$u$ is contained in the  component of~$v$ in $T\setminus e$, and with $z^{(T)}_{e,v,u}=0$ otherwise. 

The polyhedron~$Q_{\sptOp}(n)\subseteq\RR^d$  defined by the nonnegativity constraints $x\ge\zeroVec{}$,  $z\ge\zeroVec{}$,  the equations $x(E_n)=n-1$, 
\begin{equation}\label{eq:ForEF1}
	x_{\{v,w\}} - z_{\{v,w\},v,u}-z_{\{v,w\},w,u}=0\quad\text{for all pairwise distinct }v,w,u\in \ints{n}\,,
\end{equation}
as well as
\begin{equation}\label{eq:ForEF2}
	x_{\{v,w\}} + \sum_{u\in \ints{n}\setminus\{v,w\}}z_{\{v,u\},u,w}= 1
	\quad\text{for all distinct }v,w\in \ints{n}\,,
\end{equation}
satisfies $p(Q_{\sptOp}(n))=\PSpt{n}$, where~$p:\RR^d\rightarrow\RR^E_n$ is the orthogonal projection onto the $x$-variables.  This follows  from observing that, for each spanning tree $T\subseteq E_n$,  the vector $y^{(T)}=(x^{(T)},z^{(T)})$ satisfies~\eqref{eq:ForEF1} and~\eqref{eq:ForEF2}, and on the other hand, every nonnegative vector $y=(x,z)\in\RR_+^d$ satisfying~\eqref{eq:ForEF1} and~\eqref{eq:ForEF2} also satisfies 
$x(E(S))\le |S|-1$ for all $S\subseteq\ints{n}$ with $|S|\ge 2$. Indeed, to see the latter claim, one adds  equations~\eqref{eq:ForEF1} for all pairwise distinct $v,w,u\in S$ in order to obtain (after division by two and renaming summation indices)
\begin{equation}\label{eq:sptPr:1}
	(|S|-2)x(E(S))=\sum_{v,w\in S,v\ne w}\sum_{u\in S\setminus\{v,w\}}z_{\{v,u\},u,w}\,,
\end{equation}
where, due to~\eqref{eq:ForEF2} and $z\ge\zeroVec{}$, the right-hand side is bounded from above by
\begin{equation*}%\label{eq:sptPr:2}
	\sum_{v,w\in S,v\ne w}(1-x_{\{v,w\}})=|S|(|S|-1)-2x(E(S))\,,
\end{equation*}
which together with~\eqref{eq:sptPr:1} implies $x(E(S))\le |S|-1$.

For many other polytopes (with exponentially many facets) associated with polynomial time solvable combinatorial optimization problems polynomially sized extended formulations can be constructed as well (see, e.g., the recent survey~\cite{CCZ10}). Probably the most prominent problem in this class for which, however, no such  small formulation is known, is the matching problem. In fact, Yannakakis~\cite{Yan91} proved that no \emph{symmetric} polynomially sized extended formulation of the matching polytope exists. 

Here, \emph{symmetric} refers to the symmetric group $\symGr{n}$ of all permutations $\pi:\ints{n}\rightarrow\ints{n}$ of the node set~$\ints{n}$ of~$K_n$ acting on~$E_n$ via\footnote{For an action $G\times E\rightarrow E$ of a group~$G$ on a set~$E$ we use the notation $g.e$ for the image of $(g,e)$ under the action, since it makes many formulas easier to read.}
 $\pi.\{v,w\}=\{\pi(v),\pi(w)\}$
 for all $\pi\in\symGr{n}$ and $\{v,w\}\in E_n$. 
Clearly, this action of~$\symGr{n}$ on~$E_n$  induces an action on the set of all subsets of~$E_n$.  
For instance, this yields an action on the spanning trees of~$K_n$, and thus, on (the vertices of) $\PSpt{n}$. The extended formulation of~$\PSpt{n}$ discussed above is \emph{symmetric} in the sense that, for every~$\pi\in\symGr{n}$, replacing all indices associated with edges~$e\in E_n$ and nodes~$v\in \ints{n}$ by~$\pi.e$ and~$\pi.v$, respectively, does not change the set of constraints in the formulation. Phrased informally, all subsets of nodes of~$K_n$ of equal cardinality  play the same role in the formulation. For a general definition of symmetric extended formulations see Section~\ref{sec:symext}. 

In order to describe the main results of Yannakakis' paper~\cite{Yan91} and the contributions of the present paper, let us denote by 
\begin{equation*}
	\match{\ell}{n}=\setDef{M\subseteq E_n}{M\text{ matching in }K_n, |M|=\ell}
\end{equation*}
the set of all matchings of size~$\ell$ (a matching being a subset of edges no two of which share a node), and by
\begin{equation*}
	\PMatch{\ell}{n}=\conv\setDef{\charVec{M}\in\{0,1\}^{E_n}}{M\in\match{\ell}{n}}
\end{equation*}
the associated polytope. According to Edmonds~\cite{Edm65b} the \emph{perfect matching polytope} $\PMatch{n/2}{n}$ (for even~$n$) is described by
\begin{multline}\label{eq:PMPoly}
	\PMatch{n/2}{n}=\{x\in\RR_+^{E_n}\,:\,x(\delta(v))=1\text{ for all }v\in\ints{n},\\
	                  x(\delta(S))\ge 1\text{ for all }S\subseteq \ints{n},3\le|S|\text{ odd}\}
\end{multline}
(with $\delta(S)=\setDef{e\in E_n}{|e\cap S|=1}$ and $\delta(v)=\delta(\{v\})$). Yannakakis~\cite[Thm.1 and its proof]{Yan91} shows that there is a constant~$C>0$ such that, for every extended formulation  for~$\PMatch{n/2}{n}$ (with~$n$ even) that is symmetric in the sense above, the number of variables and constraints is at least $C\cdot \binom{n}{\lfloor n/4\rfloor}=2^{\Omega(n)}$. This in particular implies that there is no polynomial size symmetric extended formulation for the \emph{matching polytope} of~$K_n$ (the convex hulls of characteristic vectors of \emph{all} matchings in~$K_n$) of which the \emph{perfect} matching polytope is a face. 

Yannakakis~\cite{Yan91} moreover obtains a similar (maybe less surprising) result for traveling salesman polytopes. Denoting by
\begin{equation*}
	\cycl{\ell}{n}=\setDef{C\subseteq E_n}{C\text{ cycle in }K_n, |C|=\ell}
\end{equation*}
the set of all (simple) cycles of length~$\ell$ in~$K_n$, and by
\begin{equation*}
	\PCycl{\ell}{n}=\conv\setDef{\charVec{C}\in\{0,1\}^{E_n}}{C\in\cycl{\ell}{n}}
\end{equation*}
the associated polytopes, the \emph{traveling salesman polytope} is $\PCycl{n}{n}$. Suitably identifying~$\PMatch{n/2}{n}$ (for even~$n$) with a  face of~$\PCycl{3n}{3n}$, Yannakakis concludes that all symmetric extended formulations for~$\PCycl{n}{n}$ have size at least~$2^{\Omega(n)}$ as well~\cite[Thm.~2 and its proof]{Yan91}. 

Yannakakis' results in a fascinating way illuminate the  borders of our principal abilities to express combinatorial optimization problems like the matching or the traveling salesman problem by means of linear constraints. However, they only refer to linear descriptions that respect the inherent symmetries in the problems. In fact, the second open problem mentioned in the concluding section of~\cite{Yan91} is described as follows: ``We do not think that asymmetry helps much. Thus, prove that the matching and TSP polytopes cannot be expressed by polynomial size LP's without the asymmetry assumption.'' 

\volker{%
Indeed, it was shown very recently (and while this paper was under review) that the \emph{traveling salesman polytope} does also not possess any non-symmetric compact extended formulation~\cite{Fio2011}. The correpsonding question concerning the \emph{matching polytope}, however, still remains open.

}

\medskip
The  contribution of our paper is to show that, in contrast to the assumption expressed in the quotation above, asymmetry can help much, or, phrased differently, that symmetry requirements on extended formulations  indeed can  matter significantly with respect to the minimal sizes of extended formulations. Our main results are that both $\PMatch{\lfloor\log n\rfloor}{n}$ and $\PCycl{\lfloor\log n\rfloor}{n}$ do not admit symmetric extended formulations of polynomial size, while they have non-symmetric extended formulations of polynomial size (see Cor.~\ref{cor:noCompactSymMatch} and~\ref{cor:CompactMatch} for matchings, as well as Cor.~\ref{cor:noCompactSymCycl} and~\ref{cor:CompactCycl} for cycles). The corresponding theorems from which these corollaries are derived provide some more general and more precise results for~$\PMatch{\ell}{n}$ and~$\PCycl{\ell}{n}$. In order to establish the lower bounds for symmetric extensions, we adapt 
the techniques developed by Yannakakis~\cite{Yan91}. 
The constructions of the compact non-symmetric extended formulations rely on small families of perfect hash functions~\cite{AYZ95,FKS84,SS90}.

The paper is organized as follows. In Section~\ref{sec:symext}, we provide  definitions of extensions, extended formulations, their sizes, the crucial notion of a section of an extension, symmetry of an extension, and we give some auxilliary results. 
% In Section~\ref{sec:yannakakis}, we present Yannakakis' method to derive lower bounds on the sizes of symmetric extended formulations for perfect matching polytopes in a general setting. Here, we tried  to separate as much as possible those parts of the technique that do not rely on symmetry assumptions from those that do. Thus, we hope that the presentation in Section~\ref{sec:yannakakis} may also be useful in different contexts. 
In Section~\ref{sec:match} we  derive (using ideas from~\cite{Yan91}) lower bounds on the sizes of symmetric extended formulations for the polytopes~$\PMatch{\ell}{n}$ associated with cardinality restricted matchings.
 In Section~\ref{sec:matchPoly}, we then describe our non-symmetric extended formulations for these polytopes. Finally, in Section~\ref{sec:cyclePoly}  we present the results on $\PCycl{\ell}{n}$. Some remarks conclude the paper in Section~\ref{sec:rem}.

An extended abstract~\cite{KPT10} of this work has appeared in the proceedings of IPCO XIV. The present paper does not only  contain additional proofs that have been omitted in~\cite{KPT10} and a simplified and, as we believe, clearer presentation of the proof of Theorem~\ref{thm:lowerbound}, but it also slightly sharpens our lower bound results in two ways (based on the new Lemma~\ref{lem:subspaceExt}): We now prove lower bounds on the mere number of inequalities (rather than on the number of inequalities plus the number of variables) of symmetric extended formulations, and  these results now refer to the more general notion of symmetry obtained from considering arbitrary isometries instead of coordinate permutations only.

\paragraph{Acknowledgements.} 
\volker{%
We thank Christian Bey for   discussions on invariant subspaces 
and the referees for their comments, in particular for pointing out the simple probabilistic argument for the existence of small families of perfect hash functions (Theorem~\ref{thm:AYZ}).
}

% yannakakis/kmatch-intro.tex

\section{Extended Formulations, Extensions, and Symmetry}
\label{sec:symext}

\volker{%
Here, we formalize the central notions used in this paper and establish some basic results we will rely on later.
}

\begin{definition}
 An \emph{extension} of a polytope $P\subseteq\RR^m$ is a polyhedron $Q\subseteq\RR^d$ together with  an affine
 projection $p:\RR^d\rightarrow\RR^m$ with $p(Q)=P$. The \emph{size} of an extension is the number of its facets.
\end{definition}
\begin{definition}
 An \emph{extension} $Q\subseteq\RR^d$, $p:\RR^d\rightarrow\RR^m$ of a polytope $P\subseteq\RR^m$ is called a \emph{subspace extension} if~$Q$ is the intersection of an affine subspace of~$\RR^d$ and the nonnegative orthant $\RR_+^d$. 
\end{definition}

For instance, the polyhedron~$Q_{\sptOp}(n)$ defined in the Introduction is a subspace extension of the spanning tree polytope~$\PSpt{n}$. 

\begin{definition}
 A (finite) system of linear equations and inequalities whose 
\volker{%
set of solutions (together with some projection) forms 
an extension of~$P$
}  
is an \emph{extended formulation} for~$P$. 
The \emph{size} 
of an extended formulation is its number of inequalities (including nonnegativity constraints, but not equations). 
\end{definition}

Clearly, the size of an extended formulation is at least as large as the size of the extension it describes. Conversely,  every extension  is described by an extended formulation of at most its size. 

\old{
Note that we do not consider the number of equations (which can be assumed to be bounded by the dimension of the ambient space) in the definition of the \emph{size} of an extension or extended formulations. Actually, it would  be more elegant to drop the number of variables from the definition as well. In fact, for pointed extensions the number of variables of course is bounded from above by the number of facets, and it is not hard to see that from a possibly non-pointed extension of a polytope one can derive a pointed one with the same number of facets. However, it seems unclear whether one can do this whithout destroying symmetry of the extension (see also Section~\ref{sec:rem}). Therefore, as the lower bounds that we provide (the same applies to Yannakakis' bounds) also refer to the ambient dimension of the symmetric extensions, we stick to the definition of size as the sum of both the number of facets and the ambient dimension. We did not include the encoding lengths of the coefficients involved into the extensions into the definition of the size, since the lower bounds that we present do not refer to them. Instead, whenever we provide upper bounds, we explicitly state that the coefficients that are involved are bounded by a constant.  
}

Extensions or extended formulations of a family of polytopes $P\subseteq\RR^m$ (for varying~$m$) are \emph{compact} if their sizes and the encoding lengths of the coefficients needed to describe them can be bounded by a polynomial in~$m$ and the maximal encoding length of all components of all vertices of~$P$. Clearly, the extension~$Q_{\sptOp}(n)$ of~$\PSpt{n}$ from the Introduction is compact.

%An extended formulation is in \emph{equation form} if it consists of a system of equations and nonnegativity constraints on all variables. 

\begin{definition}
 For an \emph{extension} $Q\subseteq\RR^d$, $p:\RR^d\rightarrow\RR^m$ of a polytope $P\subseteq\RR^m$, the \emph{fiber} of $x\in P$ is the set $p^{-1}(x)=\setDef{y\in\RR^d}{p(y)=x}$.
\end{definition}

\begin{definition}
 For an \emph{extension} $Q\subseteq\RR^d$, $p:\RR^d\rightarrow\RR^m$ of a polytope $P\subseteq\RR^m$, a \emph{section} 
 $s:X\rightarrow Q$ is a map  that assigns to every vertex~$x$ of $P$  some point $s(x)\in Q\cap p^{-1}(x)$ in the intersection of  the extension~$Q$ and the fiber $p^{-1}(x)$.
\end{definition}

Such a section induces a bijection between~$X$ and its image $s(X)\subseteq Q$, whose inverse is given by~$p$. In the spanning tree example from the Introduction, the assignment  $\chi(T)\mapsto y^{(T)}=(x^{(T)},z^{(T)})$ defined such a section. 
Note that, in general,  sections will not be induced by affine maps. In fact, if  a section is induced by an affine map $s:\RR^{m}\rightarrow\RR^d$, then the intersection of~$Q$ with the affine subspace of~$\RR^d$ generated by~$s(X)$ is isomorphic to~$P$, thus~$Q$ has at least as many facets as~$P$.

%\vnew{
If~$s:X\rightarrow\ Q$ is a section for some extended formulation of~$P$ then, for each inequality $\scalProd{c}{y}\le\gamma$ in the formulation, we call the vector in~$\RR_+^X$ with entries $\gamma-\scalProd{c}{s(x)}$ ($x\in X$) a \emph{section slack covector}. Similarly,  any valid inequality $\scalProd{a}{x}\le\beta$  for~$P$ defines a \emph{slack covector}  in $\RR_+^X$ with entries $\beta-\scalProd{a}{x}$ ($x\in X$). One finds that every slack covector is a conic combination of the section slack covectors~ \cite[Cor.~2.5]{FKPT11}. In particular, we derive the following proposition (via the trivial direction of the Farkas-Lemma), which  follows from the fact that every subspace extension can be described by a linear system containing linear equations and non-negativity constraints only, implying that
in this case the coordinates of the section slack covectors simply correspond to the values of the section coordinate functions.

\begin{proposition}\label{prop:lambda}
	If $s:X\rightarrow Q$ is a section for a subspace extension $Q\subseteq\RR_+^d$ of $P=\conv(X)$ and $\scalProd{a}{x}\le\beta$ is valid for~$P$ then the system
	\begin{eqnarray}
		\sum_{x\in X} s_j(x)\cdot\lambda_x & \ge &  0 \quad\text{for all }j\in\ints{d}\label{eq:lambda:1}\\
		\sum_{x\in X} (\beta-\scalProd{a}{x})\cdot\lambda_x & = &  -1 \label{eq:lambda:2}
	\end{eqnarray}
	does not have any solution $\lambda\in\RR^X$.
\end{proposition}

\volker{%
Aiming to prove 
}
 that a (subspace) extension of a certain type does not exist one thus can first construct some appropriate section for such an extension for which one then exhibits, for some  inequality valid for~$P$,  a solution to~\eqref{eq:lambda:1}, \eqref{eq:lambda:2}.

\new{
 In order to define the notion of symmetry of an extension precisely, we will deal with groups of affine transformations $\pi:\RR^m\rightarrow\RR^m$. We will frequently use the notation $\pi.x=\pi(x)$ for $x\in \RR^m$ and $\pi.S=\pi(S)$ for $S\subseteq\RR^m$. 
Let us denote by $\isoms{d}$ the group of all affine isometries of $\RR^d$, i.e., the set of all maps $\upop:\RR^d\rightarrow\RR^d$ of the form $\upop(y)=Uy+u$ with an \emph{orthogonal} matrix $U\in\RR^{d\times d}$ (i.e., $U\transpose{U}=\ident{}$) and $u\in\RR^d$. 
The group~$\symGr{d}$ of all bijective maps from $\ints{d}=\{1,\dots,d\}$ to itself acts on~$\RR^d$ by coordinate permutations via
 $(\sigma.y)_j=y_{\sigma^{-1}(j)}$  for all $y\in\RR^d$, $\sigma\in\symGr{d}$, and $j\in\ints{d}$. 
Identifying 
  $\sigma\in\symGr{d}$ with the isometry defined via $y\mapsto\sigma.y$, we consider $\symGr{d}$ as a subgroup of $\isoms{d}$.

Suppose that~$P\subseteq\RR^m$ is a polytope with vertex set $X\subseteq\RR^m$, and~$G$ is a group of affine transformations $\pi:\RR^m\rightarrow\RR^m$ with $\pi.P=P$. Clearly, every $\pi\in G$ permutes the vertices of~$P$. 
Usually, $G$ will  be a subgroup of coordinate permutations of~$\RR^m$, i.e., a subgroup of $\symGr{m}$.

\begin{definition}
  An extension $Q\subseteq\RR^d$ of~$P$ with projection $p:\RR^d\rightarrow\RR^m$  is \emph{isometry-symmetric} with respect to~$G$
\volker{%
(or, for short, \emph{symmetric})%
}%
, if for every $\pi\in G$ there is an isometry $\upop_{\pi}\in\isoms{d}$
 with $\upop_{\pi}.Q = Q$
and
\begin{equation}\label{pGmap}
	p(\upop_{\pi}.y) = \pi.p(y) \quad\text{for all $y\in Q$}
\end{equation}
(see Fig.~\ref{fig:diagrams}); the extension is called \emph{coordinate-symmetric} if all these $\upop_{\pi}$ can be chosen to be from~$\symGr{d}$. 
\end{definition}

}
\begin{figure}[ht]
	\begin{tikzpicture}
		\matrix (m) [matrix of math nodes, row sep=3em, column sep=2.5em, text height=1.5ex, text depth=0.25ex] 
			{ Q &  Q \\
  			  P &  P \\ };
 		\path[->,font=\scriptsize] (m-1-1) edge node[auto] {$ \upop_{\pi} $} (m-1-2);
		\path[->,font=\scriptsize] (m-1-2) edge node[auto] {$ p $}           (m-2-2);
		\path[->,font=\scriptsize] (m-1-1) edge node[auto] {$ p $}           (m-2-1);
		\path[->,font=\scriptsize] (m-2-1) edge node[auto] {$ \pi $}         (m-2-2);
	\end{tikzpicture}
	\hspace{1cm}
	\begin{tikzpicture}
		\matrix (m) [matrix of math nodes, row sep=3em, column sep=2.5em, text height=1.5ex, text depth=0.25ex] 
			{ Q     &  Q \\
  			  X     &  X     \\ };
 		\path[->,font=\scriptsize] (m-1-1) edge node[auto] {$ \upop_{\pi} $} (m-1-2);
		\path[<-,font=\scriptsize] (m-1-2) edge node[auto] {$ s $}           (m-2-2);
		\path[<-,font=\scriptsize] (m-1-1) edge node[auto] {$ s $}           (m-2-1);
		\path[->,font=\scriptsize] (m-2-1) edge node[auto] {$ \pi $}         (m-2-2);
	\end{tikzpicture}
	\caption{Relations~\eqref{pGmap} (left) and~\eqref{weakSym} (right) from the definitions of \emph{symmetry} and \emph{weak coordinate-symmetry}, respectively.}
	\label{fig:diagrams}
\end{figure}

% Actually, for almost all parts of the paper it would even be sufficient to consider actions of a group~$G$ on the set~$X$ of vertices of~$P$ only and to require~\eqref{pGmap} just for vectors~$y\in p^{-1}(X)$ in the fibers of vertices of~$P$. The slightly more restrictive notion of symmetry, however, is convenient in order to  deduce lower bounds on the size of symmetric extensions for polytopes that are projections of faces of polytopes for which we have already established such bounds (Lemma~\ref{lem:boundForFace}). As this  enables us both to derive more general results on matching polytopes (Theorem~\ref{thm:lbgeneral}) and to 
% transfer such results to cycle polytopes (see Theorem~\ref{thm:lbgeneralcycl}), we decided to deal with the slightly stronger notion of symmetry. 

The prime examples of symmetric extensions arise from  extended formulations that ``look symmetric''.  

\begin{definition}
 An extended formulation $A^=y=b^=$, $A^{\le}y\le b^{\le}$ describing the polyhedron 
\begin{equation*}
	Q=\setDef{y\in\RR^d}{A^=y=b^=, A^{\le}y\le b^{\le}}
\end{equation*}
extending~$P\subseteq\RR^m$ as above is \emph{symmetric} (with respect to the action of~$G$ on~$P$), if for every~$\pi\in G$ there is 
\volker{%
some
}
 $\upop_{\pi}\in\symGr{d}$
 satisfying~\eqref{pGmap} and there are two permutations~$\varrho^=_{\pi}$ and~$\varrho^{\le}_{\pi}$ of the rows of $(A^=,b^=)$ and $(A^{\le}, b^{\le})$, respectively, such that the corresponding simultaneous  permutations of the columns and the rows  of the matrices $(A^=,b^=)$ and $(A^{\le}, b^{\le})$ leave them unchanged. 
\end{definition}

Clearly, in this situation the permutations~$\upop_{\pi}$ satisfy $\upop_{\pi}.Q= Q$, which implies the following.

\begin{lemma}\label{lem:symExtFormToSymExt}
	Every symmetric extended formulation defines a (coordinate-)symmetric extension.
\end{lemma}

One example of a symmetric extended formulation is the extended formulation for the spanning tree polytope described in the Introduction  (with respect to the group~$G$ of all permutations of the nodes of the complete graph).

\new{
The following lemma shows that we can restrict our attention to coordinate-symmetric subspace extensions when searching for the minimum size of any (iso\-metry-)symmetric extension of a given polytope~$P$. 
\volker{%
In particular, the minimum size of a symmetric extension of a polytope is attained by a \emph{pointed} symmetric extension. 
}
 
\begin{lemma}\label{lem:subspaceExt}
	If a polytope has an isometry-symmetric extension of size~$f$, then it has also a coordinate-symmetric subspace extension of size~$f$.
\end{lemma}

\begin{proof}
	Let $P\subseteq\RR^m$ be some polytope,
	 $Q\subseteq\RR^d$ some polyhedron with~$f$ facets, $p:\RR^d\rightarrow\RR^m$ some affine projection with $p(Q)=P$, and~$G$ some group acting on~$P$ and the set~$X$ of vertices of~$P$ such that, for every~$\pi\in G$, there is some $\upop_{\pi}\in\isoms{d}$ with $\upop_{\pi}.Q=Q$ and $p(\upop_{\pi}.y)=\pi.p(y)$ for all $y\in Q$. 
	
	We denote the 
	 affine hull of~$Q$ by~$\aff(Q)$ and by $L\subseteq\RR^d$  the linear subspace parallel to~$\aff(Q)$. The polyhedron~$Q$ has a unique (up to reordering of the inequalities) description 
	\begin{equation}\label{eq:eq:subspaceExt:uniqueRep}
		Q=\setDef{y\in\RR^d}{y\in\aff(Q),Ay\le b}
	\end{equation}
	with $A\in\RR^{f\times d}$ and $b\in\RR^f$,
	if we require, for each $i\in\ints{f}$, that $\row{A}{i}\in L$ and $\|\row{A}{i}\|=1$ hold for the $i$-th row $\row{A}{i}$ of~$A$. 
	 We define the affine map $\slackMap{Q}:\RR^d\rightarrow\RR^f$ (the \emph{slack map of~$Q$}) via $\slackMap{Q}(y)=b-Ay$ and call the image $\slackRep{Q}=\slackMap{Q}(Q)$ of~$Q$ under its slack map its \emph{slack representation}. Note that~$\slackRep{Q}$ is the intersection of the nonnegative orthant~$\RR_+^f$ with the affine subspace $\slackMap{Q}(\aff(Q))$. 
	
	The \emph{lineality space} $\lineal{Q}=L\cap\ker(A)$ of~$Q$ is the space of all directions of lines contained in~$Q$. As $P=p(Q)$ is a polytope (thus bounded), we find
	\begin{equation}\label{eq:subspaceExt:plinealQ}
		p(y+r)=p(y)
		\quad
		\text{for all }y\in Q, r\in\lineal{Q}\,.
	\end{equation}		
	The restriction $\slackMapRestr{Q}$  of~$\slackMap{Q}$ to the intersection of $\aff(Q)$ with the orthogonal complement $\lineal{Q}^{\perp}$ of the lineality space of~$Q$  is a bijection
	\begin{equation*}
		\slackMapRestr{Q}:\aff(Q)\cap\lineal{Q}^{\perp}\rightarrow\aff(\slackRep{Q})
	\end{equation*} 
	with
	\begin{equation}\label{eq:subspaceExt:invUpToLineal}
		(\slackMapRestr{Q})^{-1}(\slackMap{Q}(y))-y\in\lineal{Q}
		\quad
		\text{for all }y\in Q\,,
	\end{equation}
	which in particular implies
	\begin{equation}\label{eq:subspaceExt:Q}
		Q=(\slackMapRestr{Q})^{-1}(\slackRep{Q})+\lineal{Q}\,.
	\end{equation}
	It suffices to prove that~$\slackRep{Q}$ is a coordinate-symmetric extension of~$P$ via the affine projection $\tilde{p}=p\circ(\slackMapRestr{Q})^{-1}$ (which, of course, is defined on $\aff{(\slackRep{Q})}\subseteq\RR^f$ only, but can  be extended arbitrarily to~$\RR^f$ in order to formally satisfy the conditions of the definition of an extension). From~\eqref{eq:subspaceExt:Q} and~\eqref{eq:subspaceExt:plinealQ} we deduce
\begin{equation*}
	\tilde{p}(\slackRep{Q})=p((\slackMapRestr{Q})^{-1}(\slackRep{Q}))=p(Q)=P\,.
\end{equation*}
Therefore, we only need to exhibit, for each $\pi\in G$, some $\tilde{\upop}_{\pi}\in\symGr{f}$ with
\begin{equation}\label{eq:eq:subspaceExt:goalA}
	\tilde{\upop}_{\pi}.\slackRep{Q}=\slackRep{Q}
\end{equation} 
and 
\begin{equation}\label{eq:eq:subspaceExt:goalB}
	\tilde{p}(\tilde{\upop}_{\pi}.z)=\pi.\tilde{p}(z)\quad\text{for all }z\in\slackRep{Q}\,.
\end{equation} 
We construct~$\tilde{\upop}_{\pi}$ from the map $\upop_{\pi}\in\isoms{d}$ with
\begin{equation}\label{eq:eq:subspaceExt:A}
	\upop_{\pi}.Q=Q
\end{equation} 
and 
\begin{equation}\label{eq:eq:subspaceExt:B}
	p(\upop_{\pi}.y)=\pi.p(y)\quad\text{for all }y\in Q
\end{equation} 
guaranteed to exist by the symmetry of the extension~$Q$ of~$P$. Let $U\in\RR^{d\times d}$ be the orthogonal matrix and~$u\in\RR^d$ the vector with $\upop_{\pi}.y=Uy+u$ for all $y\in\RR^d$. From~\eqref{eq:eq:subspaceExt:A} (which implies $\aff(Q)=\upop_{\pi}.\aff(Q)$) and~\eqref{eq:eq:subspaceExt:uniqueRep} we derive
\begin{eqnarray*}
	Q=\upop_{\pi}^{-1}Q 
	&=& \setDef{\upop_{\pi}^{-1}.y}{y\in\aff(Q),Ay\le b}\\
	&=& \setDef{y'\in\RR^f}{\upop_{\pi}.y'\in\aff(Q),A(\upop_{\pi}.y')\le b}\\
	&=& \setDef{y'\in\RR^f}{y'\in\aff(Q),(AU)y'\le b-Au}\,.
\end{eqnarray*}
Since~$U$ is orthogonal and due to $\upop_{\pi}^{-1}.\aff(Q)=\aff(Q)$ (implying $\transpose{U}\ell\in L$ for all $\ell\in L$), the rows of the matrix~$AU$ are contained in~$L$ and have length one (since so do the rows of~$A$). Thus, because of the uniqueness of the representation~\eqref{eq:eq:subspaceExt:uniqueRep}, there is a permutation $\sigma\in\symGr{f}$ with
\begin{equation}\label{eq:eq:subspaceExt:sigma}
	(AU)_{i,\star}=A_{\sigma^{-1}(i),\star}
	\quad\text{and}\quad
 	(b-Au)_i = b_{\sigma^{-1}(i)}
\end{equation}
for all $i\in\ints{f}$. In order to show that $\tilde{\upop}_{\pi}=\sigma$ satisfies~\eqref{eq:eq:subspaceExt:goalA} and~\eqref{eq:eq:subspaceExt:goalB}, we use the equation
\begin{equation}\label{eq:eq:subspaceExt:slackMap}
	\slackMap{Q}(\upop_{\pi}.y)=\sigma.\slackMap{Q}(y)
\end{equation}
for all $y\in\RR^d$, which follows readily from
$\slackMap{Q}(\upop_{\pi}.y)=(b-Au)-(AU)y$, $\slackMap{Q}(y)=b-Ay$, and~\eqref{eq:eq:subspaceExt:sigma}. 

For each $y\in Q$ equation~\eqref{eq:eq:subspaceExt:slackMap} implies $\sigma.\slackMap{Q}(y)\in\slackRep{Q}$ due to $\upop_{\pi}.y\in Q$. Thus we conclude $\sigma.\slackRep{Q}\subseteq \slackRep{Q}$, and hence $\sigma.\slackRep{Q}=\slackRep{Q}$ since $z\mapsto\sigma.z$ defines an isometry. Thus, \eqref{eq:eq:subspaceExt:goalA} is established for $\tilde{\upop}_{\pi}=\sigma$. In  order to also show~\eqref{eq:eq:subspaceExt:goalB} for this choice of $\tilde{\upop}_{\pi}$  it remains to prove
\begin{equation}\label{eq:subspaceExt:finish}
	p((\slackMapRestr{Q})^{-1}(\sigma.\slackMap{Q}(y)))=\pi.p((\slackMapRestr{Q})^{-1})(\slackMap{Q}(y)))
\end{equation}
for all $y\in Q$. 

Due to~\eqref{eq:eq:subspaceExt:slackMap}, \eqref{eq:eq:subspaceExt:A}, and \eqref{eq:subspaceExt:invUpToLineal} the left-hand-side of~\eqref{eq:subspaceExt:finish} evaluates to $p(\upop_{\pi}.y+r)$ for some $r\in \lineal{Q}$, and thus, 
due to~\eqref{eq:subspaceExt:plinealQ} and~\eqref{eq:eq:subspaceExt:B}, to  $\pi.p(y)$. 
Similarly, the right-hand-side of~\eqref{eq:subspaceExt:finish} evaluates to $\pi.p(y)$ as well, which concludes the proof.
\end{proof}
}

% 
% For the proof of the central result on the non-existence of certain coordinate-symmetric subspace extensions (Theorem~\ref{thm:lowerbound}), only a certain aspect of coordinate-symmetry will be important that defines 
% a weaker notion of coordinate-symmetry.  
% We call  
%  an extension~$Q$ of a polytope~$P$ with vertex set~$X$ and projection~$p$ \emph{weakly coordinate-symmetric}
% \marginpar{Changed notion to \emph{weakly coordinate-symmetric} and adapted lemma below}
%  (with respect to the action of a group~$G$ on~$P$ and~$X$)
% if there is a section $s:X\rightarrow Q$  for which the  action of~$G$  on~$s(X)$ induced by the bijection~$s$ works by permutation of variables, i.e., for every $\pi\in G$ there is a permutation $\upop_{\pi}\in\symGr{d}$
% with
% \begin{equation}\label{weakSym}
% 	s(\pi.x)=\upop_{\pi}.s(x) \quad\text{for all $x\in X$}
% \end{equation}
% (see Fig.~\ref{fig:diagrams}).
% The following statement (and its proof) generalizes the construction of sections for coordinate-symmetric extensions of matching polytopes described in Yannakakis paper~\cite[Claim~1 in the proof of Thm.~1]{Yan91}. Note that within both the statement and the proof of the lemma one could simply replace \emph{coordinate-symmetric} by \emph{isometry-symmetric} everywhere, but the resulting statement would not be of  use within our context. 

The following lemma shows that coordinate-symmetric extensions have sections of a special type that will be crucial  for the proof of the central result on the non-existence of certain coordinate-symmetric subspace extensions (Theorem~\ref{thm:lowerbound}).

\begin{definition}
 A section $s:X\rightarrow Q$ for   
 an extension~$Q$ of a polytope~$P$ with vertex set~$X$ and projection~$p$   is called \emph{coordinate-symmetric} if the  action of~$G$  on~$s(X)$ induced by the action of the group~$G$ on~$X$ works by permutation of variables, i.e., if for every $\pi\in G$ there is a permutation $\upop_{\pi}\in\symGr{d}$
with
\begin{equation}\label{weakSym}
	s(\pi.x)=\upop_{\pi}.s(x) \quad\text{for all $x\in X$}
\end{equation}
(see Fig.~\ref{fig:diagrams}).
\end{definition}

The following statement (and its proof) generalizes the construction of sections for coordinate-symmetric extensions of matching polytopes described in Yannakakis paper~\cite[Claim~1 in the proof of Thm.~1]{Yan91}.

\begin{lemma}\label{lem:suffCondSymExt}
	Every coordinate-symmetric extension admits a  coordinate-symmetric section. 
\end{lemma}

\begin{proof}
Let us first observe that a coordinate-symmetric extension (with notations  as above) satisfies
\begin{equation}\label{eq:permuteFibers}
	\upop_{\pi}.p^{-1}(x)=p^{-1}(\pi.x)\quad\text{for all $\pi\in G$ and $x\in X$}\,,
\end{equation}
(thus, $\upop_{\pi}$ permutes the fibers of points in~$X$ according to~$\pi$)
since~\eqref{pGmap} readily implies $\upop_{\pi}.p^{-1}(x)\subseteq p^{-1}(\pi.x)$, from which equality follows because both sets are affine subspaces of equal dimension (as all non-empty fibers of~$p$ have the same dimension and $\upop_{\pi}.p^{-1}(x)$ is an image of one of these fibers under a bijective affine transformation).

Let~$\tilde{G}$ be the subgroup of~$\symGr{d}$
generated by $\setDef{\upop_{\pi}}{\pi\in G}$. Clearly, we have 
\begin{equation}\label{eq:GtildeQtoQ}
	\sigma.Q= Q\quad\text{for all }\sigma\in \tilde{G}\,.
\end{equation}
We start the construction of a coordinate-symmetric section~$s:X\rightarrow Q$  by choosing from each orbit $\setDef{\sigma.x}{\sigma\in G}$, $x\in X$ under the action of~$G$ some $x^{\star}\in X$ as well as an arbitrary point~$y^{\star}\in Q\cap p^{-1}(x^{\star})$ in the intersection of~$Q$ and the fiber of~$x^{\star}$. Actually, as we can consider the orbits one by one here, we will assume in the following  that there is just one of them, i.e., the action of~$G$ on~$X$ is \emph{transitive}.
Denoting by 
\begin{equation*}
	\tilde{S}(x^{\star})=\setDef{\sigma\in\tilde{G}}{\sigma.p^{-1}(x^{\star})=p^{-1}(x^{\star})}\,,
\end{equation*}
the subgroup of~$\tilde{G}$ containing all permutations that map the fiber $p^{-1}(x^{\star})$ to itself, we define 
\begin{equation}\label{eq:sofxstar}
	s(x^{\star})=\frac{1}{|\tilde{S}(x^{\star})|}\sum_{\sigma\in\tilde{S}(x^{\star})}\sigma.y^{\star}\,,
\end{equation}
which  is a point in the convex set (polyhedron) $Q\cap p^{-1}(x^{\star})$, because due to~\eqref{eq:GtildeQtoQ} we have $\sigma.y^{\star}\in Q\cap p^{-1}(x^{\star})$ for all $\sigma\in\tilde{S}(x^{\star})$.
For each $x\in X$ we now choose some $\tau_x\in G$ with $\tau_x.x^{\star}=x$ (recall that we assumed the action of~$G$ on~$X$ to be transitive) and define 
\begin{equation*}
	s(x)=\upop_{\tau_x}.s(x^{\star})\,,
\end{equation*}
which  is contained in $Q\cap p^{-1}(x)$ due to~\eqref{eq:GtildeQtoQ} and~\eqref{eq:permuteFibers}. 

In order to finish the proof of the lemma, it suffices to show $s(\pi.x)=\upop_{\pi}.s(x)$ for every $x\in X$ and~$\pi\in G$. To deduce this equation, observe that due to~\eqref{eq:permuteFibers} we have
\begin{multline*}
	\upop_{\tau_{\pi.x}}^{-1}\upop_{\pi}\upop_{\tau_x}.p^{-1}(x^{\star})=
	\upop_{\tau_{\pi.x}}^{-1}.(\upop_{\pi}.(\upop_{\tau_x}.p^{-1}(x^{\star})))\\
	=\upop_{\tau_{\pi.x}}^{-1}.(\upop_{\pi}.p^{-1}(x))
	=\upop_{\tau_{\pi.x}}^{-1}.p^{-1}(\pi.x)
	=p^{-1}(x^{\star})\,.
\end{multline*}
Thus, $\upop=\upop_{\tau_{\pi.x}}^{-1}\upop_{\pi}\upop_{\tau_x}\in\tilde{S}(x^{\star})$ holds, and in particular, $\sigma\mapsto \upop\sigma$ defines a bijection $\tilde{S}(x^{\star})\rightarrow\tilde{S}(x^{\star})$. Therefore, we can conclude 
\begin{equation}\label{eq:omega}
	\upop_{\tau_{\pi.x}}^{-1}\upop_{\pi}\upop_{\tau_x}.s(x^{\star})=\upop.s(x^{\star})
	=\frac{1}{|\tilde{S}(x^{\star})|}\sum_{\sigma\in\tilde{S}(x^{\star})}\upop\sigma.y^{\star}
	=	s(x^{\star})
\end{equation}
 from~\eqref{eq:sofxstar}, which implies the equation
\begin{equation*}
	\upop_{\pi}.s(x)=\upop_{\pi}.(\upop_{\tau_x}.s(x^{\star}))=\upop_{\pi}\upop_{\tau_x}.s(x^{\star})=\upop_{\tau_{\pi.x}}.s(x^{\star})=s(\pi.x)
\end{equation*} 
that we needed to establish.
\end{proof}

If $s:X\rightarrow Q$ is a coordinate symmetric section, then~$G$ acts on the set $\mathcal{S}=\{s_1,\dots,s_d\}$ of the component functions of~$s$ via
\begin{equation*}
	\pi.s_j=s_{\upop_{\pi^{-1}}^{-1}(j)}
\end{equation*} 
for each $j\in\ints{d}$. 
In order to see that this definition indeed is well-defined (note that $s_1,\dots,s_d$ need not be pairwise distinct functions) and yields a group action, observe that, for each $j\in\ints{d}$ and $\pi\in G$, we have
\begin{equation}\label{eq:actionCompSect}
	(\pi.s_j)(x)=s_{\upop_{\pi^{-1}}^{-1}(j)}(x)=(\upop_{\pi^{-1}}.s(x))_j=s_j(\pi^{-1}.x)\quad\text{for all }x\in X\,,
\end{equation}
from which one deduces $1.s_j=s_j$ for the one-element~$1$ in~$G$ as well as $(\pi\pi').s_j=\pi.(\pi'.s_j)$ for all $\pi,\pi'\in G$.
The \emph{isotropy} group of $s_j\in\mathcal{S}$ under this action is
\begin{equation*}
	\isoGr{G}{s_j}=\setDef{\pi\in G}{\pi.s_j=s_j}\,.
\end{equation*}
From~\eqref{eq:actionCompSect} one deduces
\begin{equation}\label{eq:isoInvariant}
	s_j(x)=s_j(\pi^{-1}.x)\quad\text{for all }x\in X, \pi\in\isoGr{G}{s_j}\,.
\end{equation}

In general, it will be impossible to identify the isotropy groups~$\isoGr{G}{s_j}$ without more knowledge on the section~$s$. However, for each isotropy group $\isoGr{G}{s_j}$, one can at least bound its index~$(G:\isoGr{G}{s_j})=\sabs{G}/\sabs{\isoGr{G}{s_j}}$  in~$G$, which will allow us to identify (large) subgroups of $\isoGr{G}{s_j}$ later.

\begin{lemma}\label{lem:orderIsoGr}
		In the setting described above, we have $(G:\isoGr{G}{s_j})\le d$\,.
\end{lemma}

\begin{proof}
This follows readily from the fact that the index $(G:\isoGr{G}{s_j})$ of the isotropy group of the element $s_j\in\mathcal{S}$ under the action of~$G$ on~$\mathcal{S}$ equals the cardinality of the orbit of~$s_j$ under that action, which due to $|\mathcal{S}|\le d$, clearly is bounded from above by~$d$. 	
\end{proof}

Finally, the following result will turn out to be useful in order to derive lower bounds on the sizes of symmetric extensions for one polytope from bounds for another one. 
\begin{lemma}\label{lem:boundForFace}
	Let $Q\subseteq\RR^d$ be an extension of the polytope~$P\subseteq\RR^m$ with projection $p:\RR^d\rightarrow\RR^m$, and let the face~$P'$ of~$P$ be an extension of  a polytope $R\subseteq\RR^{k}$   with projection $q:\RR^m\rightarrow\RR^{k}$.  Then the face $Q'=p^{-1}(P')\cap Q\subseteq\RR^d$ of~$Q$ is an extension of~$R$ via the composed projection $q\circ p:\RR^{d}\rightarrow\RR^{k}$. 
	
	If
	 the extension~$Q$ of~$P$ is symmetric with respect to  a group~$G$, and~$H$ is a group of affine transformations  such that for every $\tau\in H$ we have $\tau.R=R$ and there is some $\pi_{\tau}\in G$ with $\pi_{\tau}.P'=P'$ and $q(\pi_{\tau}.x)=\tau.q(x)$ for all $x\in P'$, then
	the extension~$Q'$ of~$R$ is  symmetric (with respect to the action of the group~$H$).
\end{lemma}

\begin{proof}
	Due to $q(p(Q'))=q(P')=R$, the polyhedron~$Q'$ (together with the projection~$q\circ p$) clearly is an extension of~$R$.
	In order to prove the statement on the symmetry of this extension, let $\tau\in H$ be an arbitrary element of~$H$ with $\pi_{\tau}\in G$  as guaranteed to exist for~$\tau$ in the statement of the lemma, and let  $\upop_{\pi_{\tau}}\in\isoms{d}$  satisfy 
$\upop_{\pi_{\tau}}.Q=Q$
and~\eqref{pGmap}
	(as guaranteed to exist by the symmetry of the extension~$Q$ of~$P$). 
	Since, for all $y\in Q'$,  we obviously have
	\begin{equation*}
		q(p(\upop_{\pi_{\tau}}.y))=q(\pi_{\tau}.p(y))=\tau.(q(p(y)))\,,
	\end{equation*}
	it suffices to show $\upop_{\pi_{\tau}}.Q'=Q'$. As $y\mapsto\upop_{\pi_{\tau}}.y$ defines an automorphism  of~$Q$ (mapping faces of~$Q$ to faces of the same dimension), it suffices to show $\upop_{\pi_{\tau}}.Q'\subseteq Q'$. Due to $\upop_{\pi_{\tau}}.Q=Q$ this relation is implied by $\upop_{\pi_{\tau}}.p^{-1}(P')\subseteq p^{-1}(P')$, which finally follows from
	\begin{equation*}
		p(\upop_{\pi_{\tau}}.p^{-1}(P'))=\pi_{\tau}.p(p^{-1}(P'))=\pi_{\tau}.P'=P'\,.
	\end{equation*}
\end{proof}

\section{Bounds on Symmetric Extensions of~$\PMatch{\ell}{n}$}
\label{sec:match}

In this section, we prove the following result, where all crucial ideas are taken from Yannakakis' paper~\cite{Yan91} (though here presented in a different way).

\begin{theorem}\label{thm:lowerbound}
	For every~$n\ge 3$ and odd~$\ell$ with $\ell\le \tfrac{n}{2}$, there exists no coordinate-symmetric subspace extension for $\PMatch{\ell}{n}$ with at most $\binom{n}{(\ell -1)/2}$ variables (with respect to the group~$\symGr{n}$ acting via permuting the nodes of~$K_n$ as described in the Introduction). 
\end{theorem}

From Theorem~\ref{thm:lowerbound}, we can derive the following more general lower bounds.
Since we need it in the proof of the next result, and also for later reference, we state a simple fact on binomial coefficients first. 

\begin{lemma}\label{lem:binomialAsymptot}
	For each constant $b\in\NN$ there is some constant~$\beta>0$ with
	\begin{equation*}
		\binom{M-b}{N}\ge \beta\binom{M}{N}
	\end{equation*}
	for all large enough~$M\in\NN$ and $N\le \frac{M}{2}$. 
\end{lemma}

\begin{theorem}\label{thm:lbgeneral}
	There is a constant $C>0$ such that, for all~$n$ and $1\le\ell\le\tfrac{n}{2}$, the size of every extension for $\PMatch{\ell}{n}$ that is symmetric (with respect to the group~$\symGr{n}$ acting via permuting the nodes of~$K_n$ as described in the Introduction) is bounded from below by 
	\begin{equation*}
		C\cdot\binom{n}{\lfloor(\ell-1)/2\rfloor}\,.
	\end{equation*}
\end{theorem}

\begin{proof}
	For odd~$\ell$, this follows from Theorem~\ref{thm:lowerbound} using Lemma~\ref{lem:subspaceExt}. For even~$\ell$, the polytope~$\PMatch{\ell-1}{n-2}$ is (isomorphic to) a face of~$\PMatch{\ell}{n}$ defined by  $x_e= 1$ for an arbitrary edge~$e$ of~$K_n$.
	From this, as $\ell-1$ is odd (and not larger than $(n-2)/2$) with $\lfloor(\ell-2)/2\rfloor=\lfloor(\ell-1)/2\rfloor$, and due to Lemma~\ref{lem:binomialAsymptot}, the theorem follows by Lemma~\ref{lem:boundForFace}.
\end{proof}

For even~$n$ and $\ell=n/2$, Theorem~\ref{thm:lbgeneral} provides a similar bound to Yannakakis  result (see Step~2 in the proof of~\cite[Theorem~1]{Yan91}) that no  coordinate-symmetric subspace extension of the perfect matching polytope of~$K_n$ has a number of variables that is bounded by $\binom{n}{k}$ for any $k<n/4$.

  Theorem~\ref{thm:lbgeneral} in particular implies that 
the size
 of every symmetric extension for $\PMatch{\ell}{n}$ with $\Omega(\log n)\le \ell\le n/2$ 
is bounded from below by $n^{\Omega(\log n)}$, which has the following consequence (due to Lemma~\ref{lem:symExtFormToSymExt}). 
\begin{corollary}\label{cor:noCompactSymMatch}
	For $\Omega(\log n)\le \ell\le n/2$, there is no compact extended formulation for $\PMatch{\ell}{n}$ that is symmetric 
(with respect to the group~$G=\symGr{n}$ acting via permuting the nodes of~$K_n$ as described in the Introduction).  
\end{corollary}

The rest of this section is devoted to prove Theorem~\ref{thm:lowerbound}. Throughout, with $\ell=2k+1$, we assume that $Q\subseteq\RR^d$ with $d\le\binom{n}{k}$ is a  coordinate-symmetric subspace extension of~$\PMatch{2k+1}{n}$ for $4k+2\le n$. We will only consider the case $k\ge 1$, as for $\ell=1$ the theorem trivially is true (note that we restrict to $n\ge 3$). 
Coordinate-symmetry is meant with respect to the action of $G=\symGr{n}$  on $\PMatch{2k+1}{n}$ and on the set~$X$ of vertices of $\PMatch{2k+1}{n}$ as described in the introduction, and we assume~$s:X\rightarrow Q$ to be a coordinate-symmetric section as guaranteed to exist by Lemma~\ref{lem:suffCondSymExt}. 
Thus, we have 
\begin{equation*}
	X=\setDef{\charVec{M}\in\{0,1\}^{E_n}}{M\in\match{2k+1}{n}}\,,
\end{equation*}
where $\match{2k+1}{n}$ is the set of all matchings~$M\subseteq E_n$ with $|M|=2k+1$ in the complete graph~$K_n=(V,E_n)$ (with $V=\ints{n}$), and 
\begin{equation*}
	(\pi.\charVec{M})_{\{v,w\}}=\charVec{M}_{\{\pi^{-1}(v),\pi^{-1}(w)\}}
\end{equation*}
holds for all $\pi\in \symGr{n}$, $M\in\match{2k+1}{n}$, and $\{v,w\}\in E_n$.
In order to simplify notations, we will sometimes identify matchings with their characteristic vectors, e.g., we write~$s(M)$ instead of~$s(\charVec{M})$ for $M\in\match{2k+1}{n}$, and we consider the action of~$\symGr{n}$ on~$\match{2k+1}{n}$.

The proof will  proceed by constructing a solution~$\lambda\in\RR^{\match{2k+1}{n}}$ to the system~\eqref{eq:lambda:1}, \eqref{eq:lambda:2}  with respect to the inequality $x(E(V_{\star}))\le k$ (valid for $\PMatch{2k+1}{n}$) for some arbitrarily chosen subset $V_{\star}\subseteq V$ of $|V_{\star}|=2k+1$ nodes. 
 
In order to determine such a  $\lambda\in\RR^{\match{2k+1}{n}}$, 
we choose an arbitrary subset $V^{\star}\subseteq V\setminus V_{\star}$ 
of cardinality $|V^{\star}|=|V_{\star}|=2k+1$ 
disjoint from~$V_{\star}$ and denote for all 
$i\in\oddints{2k+1}=\{1,3,5,\dots,2k+1\}$
\begin{equation*}
	\mathcal{M}^{\star}_i=\setDef{M\in\match{2k+1}{n}}{M\subseteq E(V_{\star}\cup V^{\star}), |M\cap(V_{\star}:V^{\star})|=i}\,,
\end{equation*}
as well as $\mathcal{M}^{\star}=\mathcal{M}^{\star}_1\cup\mathcal{M}^{\star}_3\cup\dots\cup\mathcal{M}^{\star}_{2k+1}$ (the set of all perfect matchings on the $4k+2$ nodes in $V_{\star}\cup V^{\star}$). 
In fact, we will construct a vector $\lambda\in\RR^{\match{2k+1}{n}}$ with
\begin{equation*}
	\lambda_M=
	\begin{cases}
		\lambda_i & \text{if }M\in \mathcal{M}^{\star}_i\\
		0         & \text{if }M\not\in\mathcal{M}^{\star}
	\end{cases} 
\end{equation*}
for some values $\lambda_1,\lambda_3,\dots,\lambda_{2k+1}\in\RR$ to be determined.  

The equation~\eqref{eq:lambda:2} to be satisfied now easily reads
\begin{equation}\label{eq:lambda:2:match}
	\sum_{i\in\oddints{2k+1}}\frac{i-1}{2}\cdot|\mathcal{M}^{\star}_i|\cdot \lambda_i=-1\,,
\end{equation}
while~\eqref{eq:lambda:1}, for the time being, remains
\begin{equation}\label{eq:lambda:1:match}
	\sum_{i\in\oddints{2k+1}}\sum_{M\in\mathcal{M}^{\star}_i} s_j(M)\cdot\lambda_i  \ge   0 \quad\text{for all }j\in\ints{d}\,.
\end{equation}

We are now going to simplify~\eqref{eq:lambda:1:match} by means of~\eqref{eq:isoInvariant} by identifying suitable (large) subgroups of~$\isoGr{G}{s_j}$. 
Here, the crucial ingredient  will be a  result (formulated in Lemma~\ref{lem:Yannakakis-claim}) on subgroups of the symmetric group~$\symGr{n}$, where~$\altGr{n}\subseteq\symGr{n}$ is the \emph{alternating group} formed by all even permutations of~$\ints{n}$. This result is Claim~2 in the proof of Thm.~1 of Yannakakis paper~\cite{Yan91}. His proof (which we work out below in order to make the presentation self contained at this crucial point) relies on a theorem of Bochert's~\cite{Boc89} stating that any subgroup $U$ of~$\symGr{m}$ that acts primitively on~$\ints{m}$ 
\volker{%
(i.e. the action is transitive and there is no $Y\subseteq \ints{m}$ with $1<\sabs{Y}<m$ for which $Y\cap\sigma.Y\in\{Y,\varnothing\}$ holds for all $\sigma\in U$)  contains~$\altGr{m}$ or has index at least $\lfloor (m+1)/2\rfloor!$. 
}

In the proof of Lemma~\ref{lem:Yannakakis-claim}, we use the following estimate on products of binomial coefficients. 

\begin{lemma}\label{lem:prodBinCoeff}
	For all positive integer numbers  $a,b,c_1,\dots,c_r\in\NN\setminus\{0\}$  with
	\begin{equation}\label{eq:prodBinCoeff:1}
		\max\{c_1,\dots,c_r\}\le \max\{a,b\}
		\qquad\text{and}\qquad
		\sum_{i=1}^r (c_i-1)\le a+b-2\,,
 	\end{equation}
	we have
	\begin{equation}\label{eq:prodBinCoeff:2}
		\prod_{i=1}^r c_i!\ \le\ a!\cdot b!\,.
 	\end{equation}
	If any of the inequalities in~\eqref{eq:prodBinCoeff:1} additionally is strict, then~\eqref{eq:prodBinCoeff:2} is strict as well.
\end{lemma}

\begin{proof}
	Let $2= x_1\le x_2\le \dots\le x_p$ and $2= y_1\le y_2\le \dots\le y_q$ be the ordered sequences of non-trivial factors (appearing with their multiplicities) on the left- and right-hand-side, respectively, of~\eqref{eq:prodBinCoeff:2}  (with $p=\sum_{i=1}^r (c_i-1)$ and $q=a+b-2$). Clearly, the two sequences are of the form $(1,\dots,1,2,\dots,2,\dots,d\dots d)$ with $d=\max\{c_1,\dots,c_r\}$ for the $x$- and $d=\max\{a,b\}$ for the $y$-sequence, as well as 
	  non-increasing multiplicities from $\{1,\dots,r\}$ and $\{1,2\}$, respectively.
		Due to the second inequality in~\eqref{eq:prodBinCoeff:1}, we have $p\le q$.
	
	If $x_i\le y_i$ holds for all~$i\in\ints{p}$, the statements to prove clearly are true. Otherwise, 
	 defining~$i^{\star}=\min\setDef{i\in\ints{p}}{x_i> y_{i}}-1$ ($\infty$ if the set is empty), we 
	%In case of $i^{\star}=p$, the statements to prove clearly hold.  Otherwise, 
	%we have $x_{i^{\star}+1}>y_{i^{\star}+1}$, thus $x_{i^{\star}}+1>y_{i^{\star}+1}$, and furthermore,
	find that the multiplicity of each~$x_i$ with $i>i^{\star}$ must be one (as the multiplicities are non-increasing and the multiplicities of  $y_1,\dots,y_{i^{\star}}$ are at most two). The left-hand-side of~\eqref{eq:prodBinCoeff:2} thus equals 
		\begin{equation*}
			x_1x_2\cdots x_{i^{\star}}\cdot (x_{i^{\star}}+1)(x_{i^{\star}}+2)\cdots x_p\,,
		\end{equation*}
	and the right-hand-side of~\eqref{eq:prodBinCoeff:2} is at least
	\begin{equation*}
		y_1y_2\cdots y_{i^{\star}}\cdot (y_{i^{\star}}+1)(y_{i^{\star}}+2)\cdots \max\{a,b\}\,,
	\end{equation*}
	This proves the statements of the lemma, as  $x_p\le\max\{a,b\}$ holds, and, in case of $i^{\star}<p$, we have $x_{i^{\star}}=y_{i^{\star}}$.

% 	
% 	
% 	 with $c_{k^{\star}}>x_{i^{\star}}$ (since we have $y_{j+2}\ge y_j+1$ for all $j\in\ints{r-2}$). The left-hand-side of~\eqref{eq:prodBinCoeff:2} then equals 
% 	\begin{equation*}
% 		x_1x_2\cdots x_{i^{\star}}\cdot (x_{i^{\star}}+1)\cdots c_{k^{\star}}\,,
% 	\end{equation*}
% and the right-hand-side of~\eqref{eq:prodBinCoeff:2} is at least
% \begin{equation*}
% 	y_1y_2\cdots y_{i^{\star}}\cdot (y_{i^{\star}}+1)\cdots \max\{a,b\}\,,
%\end{equation*}

\end{proof}

\begin{lemma}\label{lem:Yannakakis-claim}
	 For each subgroup~$U$ of~$\symGr{n}$ with $(\symGr{n}:U)\le\binom{n}{k}$ for  $1\le k<\frac{n}{4}$, there is some  $W\subseteq\ints{n}$ with $|W|\le k$ such that 
	\begin{equation}\label{eq:Yannakakis-claim}
		\setDef{\pi\in
		  \altGr{n}}{\pi(v)=v\text{ for all }v\in W}\subseteq U
	\end{equation}
	holds.	
\end{lemma}

\begin{proof}
	Let~$1\le k<\frac{n}{4}$ and~$U$ 
	be a subgroup  of~$\symGr{n}$ with $(\symGr{n}:U)\le\binom{n}{k}$, i.e.,
	\begin{equation}\label{eq:Uknk}
		|U|\ge k!\cdot (n-k)!\,.
	\end{equation}
	Under the action of~$U$ on the set~$\ints{n}$, there is some orbit $B\subseteq\ints{n}$ of size $|B|\ge n-k$. Indeed, this follows from~\eqref{eq:Uknk} and  Lemma~\ref{lem:prodBinCoeff}, as, for the partitioning  $\ints{n}=B_1\cup\cdots\cup B_q$ of~$\ints{n}$ into orbits, we have $|B_1|!\cdots|B_q|!\ge |U|$. We will show that $W=\ints{n}\setminus B$ (with $|W|\le k$) satisfies~\eqref{eq:Yannakakis-claim}. 
	
	Every $\pi\in\symGr{n}$ induces two permutations $\pi_W\in\symGr{W}$ and $\pi_B\in\symGr{B}$ (where we denote by $\symGr{L}$ and $\altGr{L}$ the set of all, respectively all even,  permutations of a subset $L\subseteq\ints{n}$). With the group homomorphisms $\varphi_B:U\rightarrow \symGr{B}$ and $\varphi_W:U\rightarrow \symGr{W}$ defined via $\varphi_B(\pi)=\pi_B$ and $\varphi_W(\pi)=\pi_W$ for all $\pi\in U$, we define 
	\begin{equation*}
		F=\varphi_B(\ker(\varphi_W))\,.
	\end{equation*}
It suffices to show $\altGr{B}\subseteq F$, which in turn follows by the above mentioned theorem of Bochert's~\cite{Boc89} (see, e.g., \cite[Thm.~14.2]{Wie64}), by establishing the following two statements:
\begin{enumerate}
	\item $F$ acts primitively on~$B$.
	\item $(\symGr{B}:F)<\lfloor\tfrac{|B|+1}{2}\rfloor!$
\end{enumerate}
In order to show~(1), we first show that $\ker(\varphi_W)$ (and thus its isomorphic image~$F$) acts transitively on~$B$. 
For this, we use the following fact (see, e.g.,~\cite[Prop.~7.1]{Wie64}): If an action of a group~$G$ is primitive, then the induced action of every normal subgroup~$N$ of~$G$ with $|N|>1$ is transitive. 
Choosing $G=U$ and $N=\ker(\varphi_W)$, we find that~$G=U$ acts primitively on~$B$, since, clearly, the action of~$U$ on the orbit~$B$ is transitive, and a non-trivial decompostion of~$B$ into blocks $B=B_1\cup\cdots\cup B_r$ of imprimitivity (i.e., for each $\pi\in G$, we have $\pi.B_i=B_i$ or $\pi.B_i\cap B_i=\varnothing$) with $r\ge 2$ and $|B_1|=\cdots=|B_r|=b\ge 2$   would imply
$r!\cdot(b!)^r\ge|U|\ge k!\cdot(n-k)!$, contradicting Lemma~\ref{lem:prodBinCoeff} (due to $r,b\le\frac{|B|}{2}\le\frac{n}{2}<n-k$). As~$\ker(\varphi_W)$ is normal in~$U$ and we have $k!\cdot (n-k)!\le|U|=|\im(\varphi_W)|\cdot|\ker(\varphi_W)|$ with $|\im(\varphi_W)|\le |W|!\le k!$, we have~$|\ker(\varphi_W)|>1$. Thus~$F$ acts transitively on~$B$. Similarly to the argument used above,  a non-trivial decomposition of~$B$ into blocks of imprimitivity under the action of~$F$ would yield $r!\cdot(b!)^r\ge|F|$ with $r,b\le\frac{n}{2}<n-k$, thus $r!\cdot(b!)^r\cdot{|W|!}\ge|U|\ge k!\cdot (n-k)!$ (with $|W|\le k<n-k$), again contradicting Lemma~\ref{lem:prodBinCoeff}. Hence~(1) is established.

Hence, it remains to prove~(2). From $|U|=|F|\cdot|\im(\varphi_W)|$ we deduce $|F|\ge (n-k)!$ via~\eqref{eq:Uknk} and $|W|\le k!$. Thus, it suffices to show
\begin{equation*}
	|B|(|B|-1)\cdots(n-k+1)\ <\ \lfloor\tfrac{|B|+1}{2}\rfloor!\,.
\end{equation*}
Obviously, it suffices to establish this equation for the maximal possible cardinality~$|B|=n$ and the maximal~$k$ with $k<\frac{n}{4}$. Therefore, we have to prove
\begin{equation}\label{eq:Y-claim-final}
	n(n-1)\cdots(n-k+1)\ <\ \lfloor\tfrac{n+1}{2}\rfloor!
\end{equation}
with $n=4q+r$ (for $q,r\in\NN$, $r<4$) and 
\begin{equation*}
	k=\begin{cases}
		q   & \text{if }r\ne 0\\
		q-1 & \text{if }r=0\,.
	\end{cases}
\end{equation*} 
In both cases, we have $2k<\lfloor\tfrac{n+1}{2}\rfloor=\lfloor2q+\tfrac{r+1}{2}\rfloor$, thus the right-hand-side of~\eqref{eq:Y-claim-final} has at least twice as many non-trivial factors as the left-hand-side, which, due to $n\le 2 \lfloor\tfrac{n+1}{2}\rfloor \le 3(\lfloor\tfrac{n+1}{2}\rfloor-1)\le\dots$ (as long as the first factor does not exceed the second one) establishes~\eqref{eq:Y-claim-final}.
\end{proof}

Having established Lemma~\ref{lem:Yannakakis-claim}, 
we can now continue with the proof of Theorem~\ref{thm:lowerbound}.
As we assumed $d\le\binom{n}{k}$ (with $k<\frac{n}{4}$ due to $4k+2\le n$), Lemmas~\ref{lem:orderIsoGr} and~\ref{lem:Yannakakis-claim} imply 
that, for each~$j\in\ints{d}$, there is some subset~$V_j\subseteq V$ of nodes with $|V_j|\le k$ and
\begin{equation*}
	H_j=\setDef{\pi\in
	  \altGr{n}}{\pi(v)=v\text{ for all }v\in V_j}
	\subseteq\isoGr{\symGr{n}}{s_j}\,.
\end{equation*}
%$H_j\subseteq\isoGr{\symGr{n}}{s_j}$ for all $j\in\ints{d}$. 
Two matchings $M,M'\in\match{2k+1}{n}$  are in the same orbit under the action of the group~$H_j$ if and only if we have
\begin{equation}\label{eq:orbitsHj}
	M\cap E(V_j)=M'\cap E(V_j)\quad\text{and}\quad
	V_j\setminus V(M)=V_j\setminus V(M')\,.
\end{equation}
Indeed, it is clear that~\eqref{eq:orbitsHj} holds if we have~$M'=\pi.M$ for some permutation~$\pi\in H_j$. In turn, if~\eqref{eq:orbitsHj} holds, then there clearly is some permutation~$\pi\in\symGr{n}$ with $\pi(v)=v$ for all $v\in V_j$ and $M'=\pi.M$.
 Due to $|M|=2k+1>|V_j|$ there is some edge $\{u,w\}\in M$ with $u,w\not\in V_j$. Denoting by $\tau\in\symGr{n}$ the transposition of~$u$ and~$w$, we thus also have $\pi\tau(v)=v$ for all $v\in V_j$ and $M'=\pi\tau.M$. As one of the permutations~$\pi$ and~$\pi\tau$ is even, say~$\pi'$, we find $\pi'\in H_j$ and $M'=\pi'.M$, proving that~$M$ and~$M'$ are contained in the same orbit under the action of~$H_j$.

Together with~\eqref{eq:isoInvariant}, the characterization of the orbits of~$H_j$ via~\eqref{eq:orbitsHj} yields that we have the implication
\begin{equation*}
	M\cap E(V_j)=M'\cap E(V_j)
	\quad\Rightarrow\quad
	s_j(M)=s_j(M')
\end{equation*}
for all $j\in\ints{d}$ and $M,M'\in\mathcal{M}^{\star}$ (note that we have $V(M)=V_{\star}\cup V^{\star}$ for all $M\in\mathcal{M}^{\star}$).
Denoting by $\mathcal{A}_j$ the set of all  matchings~$A$ on $V_{\star}\cup V^{\star}$ with $A\cap E(V_{\star}\cup V^{\star}\setminus V_j)=\varnothing$ (thus, $|A|\le |V_j|\le k$) and $V_j\subseteq V(A)$ (thus $s_j(M)=s_j(M')$ for all $M,M'\in\mathcal{M}^{\star}$ with $A\subseteq M$ and $A\subseteq M'$),
we hence find values $s_j(A)$ for all $A\in\mathcal{A}_j$ such that~\eqref{eq:lambda:1:match} becomes
\begin{equation}\label{eq:lambda:1:match:reform}
	\sum_{i\in\oddints{2k+1}}\sum_{A\in\mathcal{A}_j}
		s_j(A)\cdot|\setDef{M\in\mathcal{M}^{\star}_i}{A\subseteq M}|\cdot\lambda_i
		\ge 0\,.
\end{equation}
The crucial observation now is that, for each $A\in\mathcal{A}_j$,
\begin{equation*}
	g_A(i)=|\setDef{M\in\mathcal{M}^{\star}_i}{A\subseteq M}|
\end{equation*}
can be written as $|\mathcal{M}^{\star}_i|$ times a  polynomial of degree at most $|A|\le k$. In order to see this, define
\begin{equation*}
	a_{\star}=|A\cap E(V_{\star})|,\quad
	a^{\star}=|A\cap E(V^{\star})|,\quad
	a_{\star}^{\star}= |A\cap(V_{\star}:V^{\star})|\,,
\end{equation*}
and denote by $\bar{\mathcal{A}}$ the set of all matchings $A'\subseteq E(V_{\star}\cup V^{\star})$ with
\begin{equation*}
	a_{\star}=|A'\cap E(V_{\star})|,\quad
	a^{\star}=|A'\cap E(V^{\star})|,\quad
	a_{\star}^{\star}= |A'\cap(V_{\star}:V^{\star})|\,.
\end{equation*}
For symmetry reasons, we have
\begin{eqnarray*}
	g_A(i) &=& \frac{1}{|\bar{\mathcal{A}}|}\displaystyle\sum_{A'\in\bar{\mathcal{A}}}|\setDef{M\in\mathcal{M}^{\star}_i}{A'\subseteq M}|\\
	&=&\frac{1}{|\bar{\mathcal{A}}|}\displaystyle\sum_{M\in\mathcal{M}^{\star}_i}|\setDef{A'\in\bar{\mathcal{A}}}{A'\subseteq M}|\\
	&=& |\mathcal{M}^{\star}_i|\cdot\frac{1}{|\bar{\mathcal{A}}|}
		\binom{(2k+1-i)/2}{a_{\star}}\cdot \binom{i}{a_{\star}^{\star}}\cdot \binom{(2k+1-i)/2}{a^{\star}}\,,
\end{eqnarray*}
where the product of the three binomial coefficients is a   polynomial in~$i$ of degree $a_{\star}+a_{\star}^{\star}+a^{\star}=|A|\le k$. 

Hence, the left-hand-side of~\eqref{eq:lambda:1:match:reform} equals $s_j(A)\sum_{i\in\oddints{2k+1}}f_j(i)\cdot|\mathcal{M}^{\star}_i|\lambda_i$ with a polynomial~$f_j(i)$ in~$i$ of degree at most~$k$ and $f_j(0)=0$ (note $a_{\star}^{\star}\ge 1$, thus $\binom{0}{a_{\star}^{\star}}=0$). Since the left-hand-side of~\eqref{eq:lambda:2:match} equals $\sum_{i\in\oddints{2k+1}}f_0(i)\cdot|\mathcal{M}^{\star}_i|\lambda_i$ with a polynomial~$f_0(i)$ in~$i$ of degree $1$ ($\le k$) and $f_0(0)=-1$, the following lemma finally concludes the proof (by choosing $I=\oddints{2k+1}$ and $\lambda_i=\gamma_i/|\mathcal{M}^{\star}_i|$).

\begin{lemma}
	For every subset $I\subseteq\RR$ of cardinality $|I|=k+1$ there are numbers $\gamma_i\in\RR$ ($i\in I$) such that
	\begin{equation*}
		\sum_{i\in I}f(i)\cdot\gamma_i=f(0)
	\end{equation*}
	holds for all (univariate) polynomials~$f$ of degree at most~$k$.
\end{lemma}

\begin{proof}
	Suppose $I=\{i_1,\dots,i_{k+1}\}$ and set $i_0=0$ as well as $I_0=I\cup\{i_0\}$. We may assume $0\not\in I$, as otherwise the statement of the lemma is trivial. 
	Since the vector space~$\RR^{\le k}[t]$ of all (univariate) polynomials of degree at most~$k$ has dimension~$k+1$, the image of the linear map $\RR^{\le k}[t]\rightarrow \RR^{I_0}$ defined via  $f\mapsto(f(i_0),\dots,f(i_{k+1}))$ is contained in a linear hyperplane~$H\subseteq\RR^{I_0}$. As $i_0$, \dots, $i_{k+1}$ are pairwise different and every polynomial of degree at most~$k$ is determined by its values at any choice of~$k+1$ (pairwise different) arguments, $H$ does not contain any line parallel to a coordinate axis, and hence, the normal vectors to~$H$ do not have any zero-entries. In particular,  the hyperplane has a normal vector~$\gamma\in\RR^{I_0}$ whose  $i_0$-coefficient equals~$-1$.
\end{proof}

\section{A Non-symmetric Extension for $\PMatch{\ell}{n}$}
\label{sec:matchPoly}

We shall establish the following result on the existence of extensions for cardinality restricted matching polytopes in this section.

\begin{theorem}\label{thm:sizenonsymextform}
	For all~$n$ and~$\ell$, there are extensions for $\PMatch{\ell}{n}$ whose sizes can be bounded by $2^{\Order(\ell)}n^2\log n$ (and for which the encoding lengths of the coefficients  needed to describe them can be bounded by a constant). 
\end{theorem}

In particular, Theorem~\ref{thm:sizenonsymextform} implies the following, although, according to Corollary~\ref{cor:noCompactSymMatch}, no compact symmetric extended formulations exist for~$\PMatch{\ell}{n}$  with $\ell=\Theta(\log n)$. 
\begin{corollary}\label{cor:CompactMatch}
	For all~$n$ and $\ell\le\Order(\log n)$, there are compact extended formulations for $\PMatch{\ell}{n}$.
\end{corollary}

The proof of Theorem~\ref{thm:sizenonsymextform} relies on the following result on the existence of small families of \emph{perfect-hash functions}, which is from~\cite[Sect.~4]{AYZ95}. 

\begin{theorem}[Alon, Yuster, Zwick~\cite{AYZ95}]\label{thm:AYZ}
There are maps $\phi_1,\dots,\phi_{q(n,r)}:\ints{n}\rightarrow\ints{r}$ with~$q(n,r)\le 2^{\Order(r)}\log n$ such that, for every $W\subseteq\ints{n}$ with $|W|=r$, there is some $i\in\ints{q(n,r)}$ for which the map $\phi_i$ is bijective on~$W$.
\end{theorem}

\volker{%
Actually,  based on results from~\cite{FKS84,SS90}, 
Alon, Yuster, and Zwick even show that, given an index~$i$ of one of the maps and an element $v\in [n]$, the value $\phi_i(v)$ can be computed in constant time (in the uniform cost model). 

The mere extistence of such a family  follows easily\footnote{This was brought to our attention by   the referees.} from observing that for~$\phi$ chosen uniformaly at random from all maps $[n]\to[r]$ and for any $r$-element subset~$W$ of~$[n]$, the probability that~$\phi$ is bijective on~$W$ is $\frac{r!}{r^r}$. Choosing $m := \frac{r^{r+1}}{r!} \ln n $ such maps  $\phi_1,\dots,\phi_m$ independently, for every $r$-element subset~$W$ of~$[n]$ the probability that none of $\phi_1,\dots,\phi_m$ is bijective on~$W$ is 
\begin{equation*}
 (1-\tfrac{r!}{r^r})^m \le e^{-m \frac{r!}{r^r}} \le n^{-r}.
\end{equation*}
Thus, the probability that there is 
some
$r$-element subset~$W$ of $[n]$
on which none of the~$\phi_i$ is 
bijective
is at most
\begin{equation*}
 \binom{n}{r} n^{-r} < 1.
\end{equation*}
Hence, the probability that 
$\phi_1,\dots,\phi_m$ 
have the desired property is
non-zero.
The proof is concluded by noting that
we have $\frac{r^{r+1}}{r!}=2^{\Theta(r)}$ by Stirling's formula.
}

Additionally to Theorem~\ref{thm:AYZ},
we will use the following construction of an extension of a polytope that is specified as the convex hull of some polytopes of which extensions are already available. The result essentially is due to Balas (see, e.g., \cite[Thm.2.1]{Bal85}). In the form it is stated here,  it is taken from~\cite[Cor.~3]{KL10}, where it is derived from general results on \emph{branched polyhedral systems}. Actually, 
 in this section we will need only  the special  case that the extensions used in the construction  are the polytopes themselves. However, we will face the slightly more general situation in our treatment of cycle polytopes in Section~\ref{sec:cyclePoly}. 
% Furthermore, we will use the following two auxilliary results. The first one (Lemma~\ref{lem:extensionOfUnion}) provides a construction of an extension of a polytope that is specified as the convex hull of some polytopes of which extensions are already available. In fact, 
%  in this section it will be needed only for the case that these extensions are the polytopes themselves (this is a special case of a result of Balas', see~\cite[Thm.2.1]{Bal85}). However, we will face the slightly more general situation in our treatment of cycle polytopes in Section~\ref{sec:cyclePoly}. 

\begin{lemma}\label{lem:extensionOfUnion}
	If the polytopes $P_i\subseteq\RR^m$ (for $i\in\ints{q}$) have extensions~$Q_i$ of size~$s_i$, respectively, then
	\begin{equation*}
		P=\conv(P_1\cup\cdots\cup P_q)
	\end{equation*}
	has an extension of size $\sum_{i=1}^q(s_i+1)$.
\end{lemma}

In order to prove Theorem~\ref{thm:sizenonsymextform}, let 
 $\phi_1,\dots,\phi_q$ be maps as guaranteed to exist by Theorem~\ref{thm:AYZ} with $r=2\ell$ and $q=q(n,2\ell)\le 2^{\Order(\ell)}\log n$, and denote 
	$\mathcal{M}_i=\setDef{M\in\match{\ell}{n}}{\phi_i\text{ is bijective on }V(M)}$
for each $i\in\ints{q}$. By Theorem~\ref{thm:AYZ}, we have $\match{\ell}{n}=\mathcal{M}_1\cup\cdots\cup\mathcal{M}_q$. Consequently, 
\begin{equation}\label{eq:convUnion}
	\PMatch{\ell}{n}=\conv(P_1\cup\cdots\cup P_q)
\end{equation}
with $P_i=\conv\setDef{\charVec{M}}{M\in\mathcal{M}_i}$ 
for all $i\in\ints{q}$. Using the concept of branched polyhedral systems mentioned above along with Edmonds'  Matching Theorem~\cite{Edm65b} (see~\eqref{eq:PMPoly}), one finds (see~\cite[Sect.~4.4]{KL10} 
\volker{for the derivation})
\begin{multline*}
	P_i=\{x\in\RR_+^{E_n}\,:\,x_{E_n\setminus E^i}=\zeroVec{}, x(\delta(\phi_i^{-1}(s)))= 1\text{ for all }s\in\ints{2\ell},\\
%	                             x(E_i(\phi_i^{-1}(S)))\le (|S|-1)/2 \text{ for all }S\subseteq\ints{2\ell}, |S|\text{ odd}\}
	                             x(\delta(\phi_i^{-1}(S)))\ge 1 \text{ for all }S\subseteq\ints{2\ell}, |S|\text{ odd}\}	\,,
\end{multline*}
where $E^i=E_n\setminus\bigcup_{j\in\ints{2\ell}}E(\phi_i^{-1}(j))$.
%This follows by Lemma~\ref{lem:expand} from Edmonds' linear description~\eqref{eq:PMPoly} of the perfect matching polytope~$\PMatch{\ell}{2\ell}$ of $K_{2\ell}$.
As the sum of the number of variables and the number of inequalities in the description of~$P_i$ is bounded by $2^{\Order(\ell)}+n^2$ (the summand~$n^2$ coming from the nonnegativity constraints on~$x\in\RR_+^{E_n}$ and the constant in $\Order(\ell)$ being independent of~$i$), we obtain an extension of~$\PMatch{\ell}{n}$ of size  $2^{\Order(\ell)}n^2\log n$ by Lemma~\ref{lem:extensionOfUnion}. This proves Theorem~\ref{thm:sizenonsymextform}.
%%%%%%%%%%%%%%%%%%%%%%%%%%%%%%%%%%%%%%%%%%%%%%%%%%%%%%%%%%%%%%%%%%%%%%%%%%%%%%%%%%%%%%%%%%%%%%%%%%%%%%%%%%%%%%%%%%%%%%%%%%%%%%%%%%%%%%%%%%%%%%%%%%%%%%

\section{Extensions for Cycle Polytopes}
\label{sec:cyclePoly}

By a  modification of Yannakakis'  construction for the derivation of lower bounds on the sizes of symmetric extensions for  traveling salesman polytopes from the corresponding lower bounds for matching polytopes \cite[Thm.~2]{Yan91}, we  obtain lower bounds on the sizes of symmetric extensions for $\PCycl{\ell}{n}$. The lower bound $\ell\ge 42$ in the statement of the theorem is convenient with respect to both formulating the bound and 
proving its validity. 

\begin{theorem}\label{thm:lbgeneralcycl}
	There is a constant $C'>0$ such that, for all~$n$ and $42\le\ell\le n$, the size of every extension for $\PCycl{\ell}{n}$ that is symmetric (with respect to the group~$\symGr{n}$ acting via permuting the nodes of~$K_n$ as described in the Introduction) is bounded from below by 
	\begin{equation*}
		C'\cdot\binom{\lfloor\frac{n}{3}\rfloor}{\lfloor(\lfloor\frac{\ell}{6}\rfloor-1)/2\rfloor}\,.
	\end{equation*}
\end{theorem}

\begin{proof}
	For $\ell\le n$, let us define $\bar{\ell}\in\{0,\dots,5\}$ and $n',\ell'\in\NN$ via
	\begin{equation*}
		\bar{\ell}=\ell\ \text{mod } 6
	    \ ,\quad
		n'=\lfloor\frac{n-\bar{\ell}}{3}\rfloor
		\ ,\quad\text{and}\quad
		\ell'=\lfloor\frac{\ell}{6}\rfloor=\frac{\ell-\bar{\ell}}{6}\,.
	\end{equation*} 
	For later reference, let us  argue that we have
	\begin{equation}\label{eq:lbgeneralcycl:prime}
		\ell'\le\frac{n'}{2}\,.
	\end{equation}
	In order to establish~\eqref{eq:lbgeneralcycl:prime}, we have to show
	\begin{equation}\label{eq:lbgeneralcycl:prime2}
		\frac{\ell-\bar{\ell}}{3}\le\lfloor\frac{n-\bar{\ell}}{3}\rfloor\,,
	\end{equation}
	which follows readily for $\ell\le n-2$ (due to $\lfloor a/3\rfloor\ge (a-2)/3$ for all $a\in\ZZ$). For $\ell\ge n-2$ (thus $0\le n-\ell\le 2$) we have
	\begin{equation*}
		(n-\bar{\ell})\text{ mod }3=
		((n-\bar{\ell})\text{ mod }6)\text{ mod }3=
		(n-\ell)\text{ mod }3=
		n-\ell\,,
	\end{equation*}
	and thus~\eqref{eq:lbgeneralcycl:prime2} in this case is satisfied due to
	\begin{equation*}
		\lfloor\frac{n-\bar{\ell}}{3}\rfloor
		=\frac{n-\bar{\ell}}{3}-\frac13((n-\bar{\ell})\text{ mod }3)
		=\frac{n-\bar{\ell}}{3}-\frac{n-\ell}{3}
		=\frac{\ell-\bar{\ell}}{3}\,.
	\end{equation*}

	As we have $3n'+\bar{\ell}\le n$, we can find four pairwise disjoint subsets~$S$, $T$, $R$, and~$U$ of nodes of the complete graph~$K_n=(V,E_n)$ on~$n$ nodes with  $|S|=|T|=|U|=n'$ and $|R|=\bar{\ell}$ (see Fig.~\ref{fig:MatchCycl}). We denote the elements of these sets as follows:
	\begin{equation*}
		S=\{s_1,\dots,s_{n'}\}\quad
		T=\{t_1,\dots,t_{n'}\}\quad
		U=\{u_1,\dots,u_{n'}\}\quad%\text{and}\quad
		R=\{r_1,\dots,r_{\bar{\ell}}\}
	\end{equation*}
	Define the subset
	\begin{multline*}
		E^0=(S:U)\cup (S:R)\cup\setDef{\{t_i,v\}\in E_n}{i\in\ints{n'}, v\in V\setminus\{s_i,u_i\}}
			  %\cup (\{r_2,\dots,r_{\bar{\ell}-1}\}:U)
	\end{multline*}
	% and 
	% \begin{equation*}
	% 	E_1=\{\{r_1,r_2\},\{r_2,r_3\},\dots,\{r_{\bar{\ell}-1},r_{\bar{\ell}}\}\}
	% \end{equation*}
	of edges of~$K_n$, and 
	denote by $F$ the face of $\PCycl{\ell}{n}$ that is defined by $x_e=0$ for all $e\in E^0$. 
	
	Every cycle $C\in\cycl{\ell}{n}$ with $C\cap E^0=\varnothing$ satisfies $|V(C)\cap T|\le 2\lfloor\ell/6\rfloor$, because~$C$ visits  at least two nodes (from $V\setminus T$) between any two visits to~$T$, and $|V(C)\cap T|$ is even. 
Therefore, denoting
\begin{equation*}
	\tilde{\mathcal{C}}=\setDef{C\in\cycl{\ell}{n}}{C\cap E^0=\varnothing, |V(C)\cap T|=2\lfloor\ell/6\rfloor}\,,
\end{equation*}
we find that 
\begin{equation*}
	\tilde{F}=\conv\setDef{\charVec{C}}{C\in\tilde{\mathcal{C}}}
	=\setDef{x\in F}{x(\delta(T))=4\lfloor\ell/6\rfloor}
\end{equation*}
is a face of~$F$. Moreover, for every $C\in\tilde{\mathcal{C}}$, we have $|C\cap E(S)|\ge \lfloor\ell/6\rfloor$. Thus, with
\begin{equation*}
	\mathcal{C}'=\setDef{C\in\tilde{\mathcal{C}}}{|C\cap E(S)|=\lfloor\ell/6\rfloor}
\end{equation*}
we find that 
\begin{equation*}
	P'=\conv\setDef{\charVec{C}}{C\in\mathcal{C}'}=\setDef{x\in \tilde{F}}{x(E(S))=\lfloor\ell/6\rfloor}
\end{equation*}
is a face of~$\tilde{F}$. It is the face 
\begin{equation*}
	P'=\setDef{x\in \PCycl{\ell}{n}}{x(E(S))=\lfloor\ell/6\rfloor,x(\delta(T))=4\lfloor\ell/6\rfloor, x_{E^0}=\zeroVec{}}
\end{equation*} 
of~$\PCycl{\ell}{n}$.

\begin{figure}[ht]\label{fig:MatchCycl}
	\centering
	\includegraphics[height=6cm]{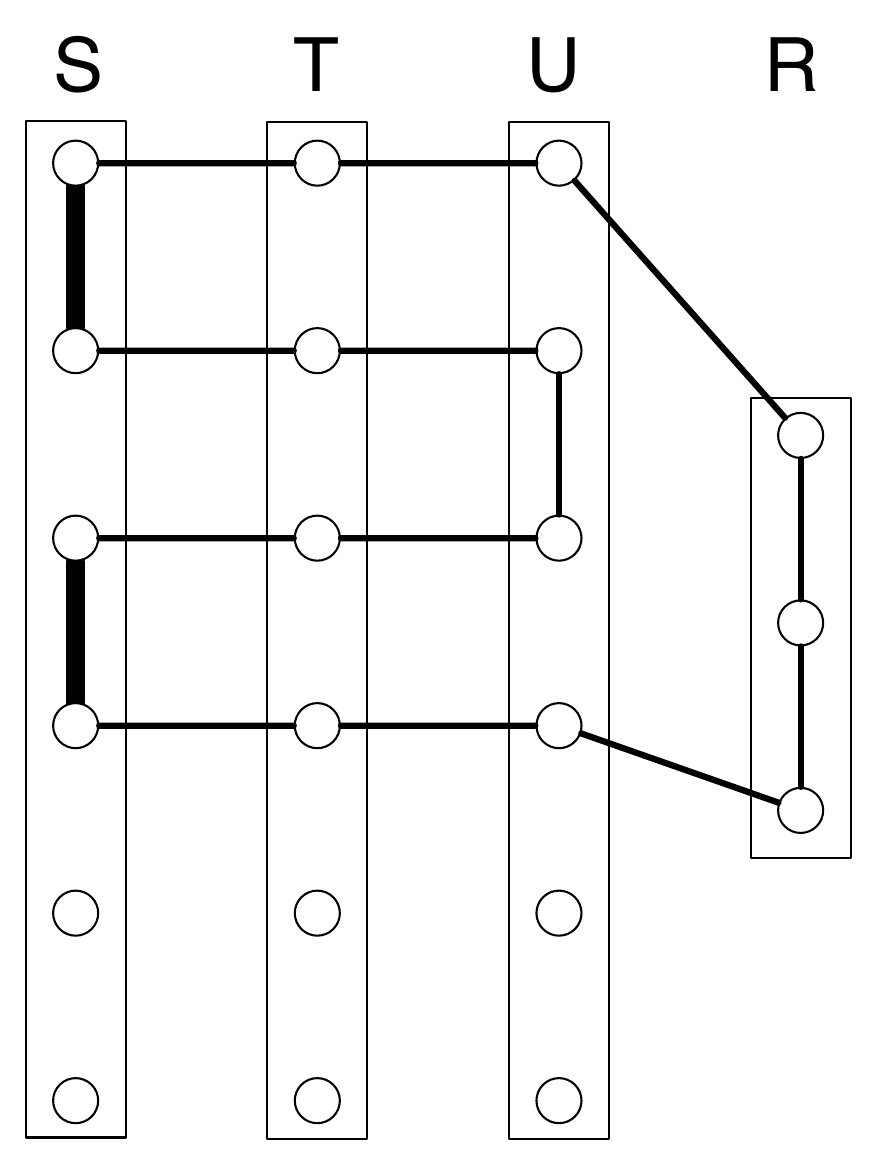}
	\caption{A cycle of length $\ell=15$ in~$K_{21}$ inducing a matching of size~$2$ in~$K_5$.}
\end{figure}

Since a cycle $C\in\cycl{\ell}{n}$ is contained in~$\mathcal{C}'$ if and only if $C\cap E(S)$ is a matching of size~$\ell'=\lfloor\ell/6\rfloor$, we find that via the orthogonal projection $q:\RR^{E_n}\rightarrow\RR^{E(S)}$ we have  
\begin{equation*}
	q(P')=\PMatch{\ell'}{n'}
\end{equation*}
after identification of~$S$ with the node set of $K_{n'}$ via $s_i\mapsto i$ for all $i\in\ints{n'}$. Moreover, for every $\tau\in\symGr{n'}$ the permutation $\pi\in\symGr{n}$ with
\begin{equation*}
	\pi(s_i)=s_{\tau(i)}, \quad
	\pi(t_i)=t_{\tau(i)}, \quad
	\pi(u_i)=u_{\tau(i)}
\end{equation*}
for all $i\in\ints{n'}$, and $\pi(r)=r$ for all $r\in R$ satisfies $\pi.P'=P'$ and
\begin{equation*}
	q(\pi.x)=\tau.q(x)\quad\text{for all }x\in\RR^{E_{n'}}\,.
\end{equation*}
Hence, due to Lemma~\ref{lem:boundForFace}, a symmetric extension of~$\PCycl{\ell}{n}$ of size~$s$ yields a symmetric extension of~$\PMatch{\ell'}{n'}$ of size at most~$s+n^2$ (as one can define the face~$P'$ of~$\PCycl{\ell}{n}$ by 
$2+|E^0|\le n^2$ equations), which, due to~\eqref{eq:lbgeneralcycl:prime} and Theorem~\ref{thm:lbgeneral} implies (with the constant~$C>0$ from Theorem~\ref{thm:lbgeneral})
	\begin{equation}\label{eq:lbgeneralcycl:s}
		s\ge \frac{C}{2}\cdot\binom{\lfloor\frac{n-\bar{\ell}}{3}\rfloor}{\lfloor(\lfloor\frac{\ell}{6}\rfloor-1)/2\rfloor}
	\end{equation}
for large enough~$n$ (since, due to~$\ell\ge 42$, the binomial coefficient in~\eqref{eq:lbgeneralcycl:s} grows at least cubically in~$n$). Because of $\bar{\ell}\le 5$, Lemma~\ref{lem:binomialAsymptot} implies the existence of a constant~$C'>0$ as claimed in the theorem.
\end{proof}

\begin{corollary}\label{cor:noCompactSymCycl}
	For $\Omega(\log n)\le \ell\le n$, there is no compact extended formulation for $\PCycl{\ell}{n}$ that is symmetric 
(with respect to the group~$\symGr{n}$ acting via permuting the nodes of~$K_n$ as described in the Introduction).  
\end{corollary}

On the other hand, if we drop the symmetry requirement, we find extensions of the following size.

\begin{theorem}\label{thm:sizenonsymextformcycl}
	For all~$n$ and~$\ell$, there are extensions for $\PCycl{\ell}{n}$ whose sizes can be bounded by $2^{\Order(\ell)}n^3\log n$ (and for which the encoding lengths of the coefficients  needed to describe them can be bounded by a constant).
\end{theorem}
%\marginpar{Alternatively: $\Order(\ell !\cdot m\cdot \log^2n)$, where~$m$ is the number of edges (arbitrary graph)}

Before we prove Theorem~\ref{thm:sizenonsymextformcycl}, we state a consequence that is similar to Corollary~\ref{cor:noCompactSymMatch} for matching polytopes. It shows that, despite the non-existence of symmetric extensions for the polytopes associated with cycles of length $\Theta(\log n)$ (Corollary~\ref{cor:noCompactSymCycl}), there are non-symmetric compact extensions of these polytopes. 

\begin{corollary}\label{cor:CompactCycl}
	For all~$n$ and $\ell\le\Order(\log n)$, there are compact extended formulations for $\PCycl{\ell}{n}$.
\end{corollary}

The rest of the section is devoted to prove Theorem~\ref{thm:sizenonsymextformcycl}, i.e., to construct an extension of $\PCycl{\ell}{n}$ whose size is bounded by $2^{\Order(\ell)}n^3\log n$. We proceed similarly to the proof of Theorem~\ref{thm:sizenonsymextform} (the construction of extensions for matching polytopes), this time starting with   maps $\phi_1,\dots,\phi_q$ as guaranteed to exist by Theorem~\ref{thm:AYZ} with $r=\ell$ and $q=q(n,\ell)\le 2^{\Order(\ell)}\log n$, and defining 
\begin{equation*}
	\mathcal{C}_i=\setDef{C\in\cycl{\ell}{n}}{\phi_i\text{ is bijective on }V(C)}
\end{equation*}
for each $i\in\ints{q}$. Thus, we have $\cycl{\ell}{n}=\mathcal{C}_1\cup\cdots\cup\mathcal{C}_q$, and hence, 
\begin{equation}\label{eq:convUnionCycl}
	\PCycl{\ell}{n}=\conv(P_1\cup\cdots\cup P_q)
\end{equation}
with $P_i=\conv\setDef{\charVec{C}}{C\in\mathcal{C}_i}$
for all $i\in\ints{q}$. Due to Lemma~\ref{lem:extensionOfUnion}, it suffices to exhibit, for each~$i\in\ints{q}$, an extension of~$P_i$ of size bounded by $\Order(2^{\ell}\cdot n^3)$ (with the constant independent of~$i$). Towards this end, let for $i\in\ints{q}$
\begin{equation*}
	V_c=\phi_i^{-1}(c)\quad\text{for all }c\in\ints{\ell}\,,
\end{equation*}
and define, for each $v^{\star}\in V_{\ell}$,
\begin{equation*}
	P_i(v^{\star})=\conv\setDef{\charVec{C}}{C\in\mathcal{C}_i,v^{\star}\in V(C)}\,.
\end{equation*}
Thus, we have
\begin{equation*}
	P_i=\conv\bigcup_{v^{\star}\in V_{\ell}}{P_i(v^{\star})}\,,
\end{equation*}
and hence, again due to Lemma~\ref{lem:extensionOfUnion}, it suffices to construct  extensions of the $P_i(v^{\star})$, whose sizes are bounded by $\Order(2^{\ell}\cdot n^2)$. 

In order to derive such extensions define, for each $i\in\ints{q}$ and $v^{\star}\in V_{\ell}$, a directed acyclic graph~$D$ with nodes
\begin{equation*}
	(A,v)\quad\text{for all }A\subseteq\ints{\ell-1}\text{ and }v\in \phi_i^{-1}(A)\,,
\end{equation*}
as well as two additional nodes~$s$ and~$t$, and arcs 
\begin{equation*}
	\big(s,(\{\phi_i(v)\},v)\big)\quad\text{and}\quad\big((\ints{\ell-1},v),t\big)
\end{equation*} 
for all $v\in\phi_i^{-1}(\ints{\ell-1})$, as well as 
\begin{equation*}
	\big((A,v),(A\cup\{\phi_i(w)\},w)\big)
\end{equation*} 
for all $A\subseteq\ints{\ell-1}$, $v\in\phi_i^{-1}(A)$, and $w\in\phi_i^{-1}(\ints{\ell-1}\setminus A)$. This is basically the dynamic programming digraph (using an idea going back to~\cite{HK62}) from the color-coding method for finding  paths of prescribed lengths described in~\cite{AYZ95}. Each $s$-$t$-path in~$D$ corresponds to a cycle in~$\mathcal{C}_i$ that visits~$v^{\star}$, and each such cycle, in turn, corresponds to two $s$-$t$-paths in~$D$ (one for each of the two directions of transversal). 

Defining $Q_i(v^{\star})$ as the convex hull of the characteristic vectors of all $s$-$t$-paths in~$D$ in the arc space of~$D$, we find that~$P_i(v^{\star})$ is the image of $Q_i(v^{\star}))$ under the projection whose component function  corresponding to the edge $\{v,w\}$ of~$K_n$ is given by the sum of all arc variables corresponding to arcs $((A,v),(A',w))$ (for  $A,A'\subseteq\ints{\ell-1}$) if $v^{\star}\not\in\{v,w\}$, and by the sum of the two arc variables corresponding to $(s,(\{\phi_i(w)\},w))$ and $((\ints{\ell-1},w),t)$ in case of $v=v^{\star}$. Clearly, $Q_i(v^{\star})$ can be described by the nonnegativity constraints, the flow conservation constraints for all nodes in~$D$ different from~$s$ and~$t$, and by the equation stating that there must be exactly one flow-unit leaving~$s$. 
 As the number of arcs of~$D$ is bounded by $\Order(2^{\ell}\cdot n^2)$, we thus have found an extension of~$P_i(v^{\star})$ of the desired size.

\section{Conclusions}
\label{sec:rem}

The results presented in this paper demonstrate that  there are polytopes which have compact extended formulations though they do not admit compact symmetric ones. These polytopes are associated with matchings (or cycles) of some prescribed cardinalities.
Nevertheless, whether there are compact extended formulations for  general matching polytopes (or for  perfect matching polytopes) or not, remains one of the most interesting open question here. In fact, it is even unknown whether there are  any (non-symmetric) extended  formulations of these polytopes of size $2^{\text{o}(n)}$. 
\new{
In general, it is not at all well understood how small extended formulations of concrete polytopes can be.  
 One problem is that the currently available methods to bound the sizes of general extended formulations from below have rather limited power (see~\cite{FKPT11}). Note, however, that via  counting arguments one can prove, e.g., that there are 0/1-polytopes (even independence polytopes of matroids) that do not admit compact extended formulations~\cite{Rot11}.

In any case, the investigation of the  limits of the  concept of extended formulations seems to be not only a quite  relevant topic from the point of view of optimization, but it also opens many interesting connections to other branches of mathematics. Some of these have played a role in this paper, others would be the nonnegative rank of matrices and communication complexity. For details, we refer once more to Yannakakis paper~\cite{Yan91} (see also~\cite{FKPT11,Kai11}). 
}


\begin{thebibliography}{10}

\bibitem{AYZ95}
Noga Alon, Raphael Yuster, and Uri Zwick.
\newblock Color-coding.
\newblock {\em J. Assoc. Comput. Mach.}, 42(4):844--856, 1995.

\bibitem{Bal85}
Egon Balas.
\newblock Disjunctive programming and a hierarchy of relaxations for discrete
  optimization problems.
\newblock {\em SIAM J. Algebraic Discrete Methods}, 6(3):466--486, 1985.

\bibitem{Boc89}
Alfred Bochert.
\newblock Ueber die {Z}ahl der verschiedenen {W}erthe, die eine {F}unction
  gegebener {B}uchstaben durch {V}ertauschung derselben erlangen kann.
\newblock {\em Math. Ann.}, 33(4):584--590, 1889.

\bibitem{CCZ10}
Michele Conforti, G{{\'e}}rard Cornu{{\'e}}jols, and Giacomo Zambelli.
\newblock Extended formulations in combinatorial optimization.
\newblock {\em 4OR}, 8(1):1--48, 2010.

\bibitem{Edm65b}
Jack Edmonds.
\newblock Maximum matching and a polyhedron with {$0,1$}-vertices.
\newblock {\em J. Res. Nat. Bur. Standards Sect. B}, 69B:125--130, 1965.

\bibitem{Edm71}
Jack Edmonds.
\newblock Matroids and the greedy algorithm.
\newblock {\em Math. Programming}, 1:127--136, 1971.

\bibitem{FKPT11}
Samuel {Fiorini}, Volker {Kaibel}, K.~{Pashkovich}, and Dirk~Oliver. {Theis}.
\newblock {Combinatorial Bounds on Nonnegative Rank and Extended Formulations}.
\newblock {\em ArXiv e-prints}, November 2011, arXiv:1111.0444.

\bibitem{Fio2011}
Samuel Fiorini, Serge Massar, Sebastian Pokutta, Hans~Raj Tiwary, and Ronald
  de~Wolf.
\newblock Linear vs. semidefinite extended formulations: exponential separation
  and strong lower bounds.
\newblock In {\em STOC}, pages 95--106, 2012.

\bibitem{FKS84}
Michael~L. Fredman, J{{\'a}}nos Koml{{\'o}}s, and Endre Szemer{{\'e}}di.
\newblock Storing a sparse table with {$O(1)$} worst case access time.
\newblock {\em J. Assoc. Comput. Mach.}, 31(3):538--544, 1984.

\bibitem{HK62}
Michael Held and Richard~M. Karp.
\newblock A dynamic programming approach to sequencing problems.
\newblock {\em J. Soc. Indust. Appl. Math.}, 10:196--210, 1962.

\bibitem{Kai11}
Volker Kaibel.
\newblock Extended formulations in combinatorial optimization.
\newblock {\em Optima}, 85:2--7, 2011.

\bibitem{KL10}
Volker Kaibel and Andreas Loos.
\newblock Branched polyhedral systems.
\newblock In Friedrich Eisenbrand and Bruce Shepherd, editors, {\em Integer
  Programming and Combinatorial Optimization (Proc. IPCO XIV)}, volume 6080 of
  {\em LNCS}, pages 177--190. Springer, 2010.

\bibitem{KPT10}
Volker Kaibel, Kanstantsin Pashkovich, and Dirk~Oliver Theis.
\newblock Symmetry matters for the sizes of extended formulations.
\newblock In Friedrich Eisenbrand and Bruce Shepherd, editors, {\em Integer
  Programming and Combinatorial Optimization (Proc. IPCO XIV)}, volume 6080 of
  {\em LNCS}, pages 135--148. Springer, 2010.

\bibitem{Mar87}
R.~Kipp~Martin.
\newblock Using separation algorithms to generate mixed integer model
  reformulations.
\newblock Technical report, University of Chicago, 1987.

\bibitem{Rot11}
Thomas {Rothvo{\ss}}.
\newblock {Some 0/1 polytopes need exponential size extended formulations}.
\newblock {\em ArXiv e-prints}, April 2011, arXiv:1105.0036.

\bibitem{SS90}
Jeanette~P. Schmidt and Alan Siegel.
\newblock The spatial complexity of oblivious {$k$}-probe hash functions.
\newblock {\em SIAM J. Comput.}, 19(5):775--786, 1990.

\bibitem{Wie64}
Helmut Wielandt.
\newblock {\em Finite permutation groups}.
\newblock Translated from the German by R. Bercov. Academic Press, New York,
  1964.

\bibitem{Yan91}
Mihalis Yannakakis.
\newblock Expressing combinatorial optimization problems by linear programs.
\newblock {\em J. Comput. System Sci.}, 43(3):441--466, 1991.

\end{thebibliography}
\end{document}